\documentclass[12pt]{amsart}

\usepackage{euscript,amssymb,xr}

\let\oldmarginpar\marginpar
\renewcommand\marginpar[1]
{\oldmarginpar{\tiny\bf \begin{flushleft} #1 \end{flushleft}}}

\input xy
\xyoption{all}

\newcommand{\la}{\langle}
\newcommand{\ra}{\rangle}

\newtheorem{theorem}{\bf Theorem}[section]
\newtheorem{lemma}[theorem]{\bf Lemma}

\newtheorem{corollary}[theorem]{\bf Corollary}

\newtheorem{conjecture}[theorem]{\bf Conjecture}

\newcommand{\CC}{{\Bbb C}}

\newcommand{\CP}{{\Bbb CP}}

\newcommand{\NN}{{\Bbb N}}

\newcommand{\QQ}{{\Bbb Q}}
\newcommand{\RR}{{\Bbb R}}

\newcommand{\ZZ}{{\Bbb Z}}

\newcommand{\ggreat}{>\kern-.7ex>}
\newcommand{\ssmall}{<\kern-.7ex<}
\newcommand{\qu}{/\kern-.7ex/}
\newcommand{\exh}{\to\kern-1.8ex\to}

\newcommand{\hH}{{\EuScript{H}}}

\newcommand{\kK}{{\EuScript{K}}}
\newcommand{\lL}{{\EuScript{L}}}

\newcommand{\pP}{{\EuScript{P}}}

\newcommand{\uU}{{\EuScript{U}}}
\newcommand{\vV}{{\EuScript{V}}}

\newcommand{\GL}{\operatorname{GL}}

\newcommand{\Aut}{\operatorname{Aut}}

\newcommand{\const}{\operatorname{const}}

\newcommand{\Coker}{\operatorname{Coker}}

\newcommand{\GCD}{\operatorname{GCD}}

\newcommand{\Hom}{\operatorname{Hom}}
\newcommand{\Homeo}{\operatorname{Homeo}}
\newcommand{\Id}{\operatorname{Id}}

\newcommand{\Ker}{\operatorname{Ker}}

\newcommand{\Map}{\operatorname{Map}}
\newcommand{\Min}{\operatorname{Min}}

\renewcommand{\O}{\operatorname{O}}

\newcommand{\Tor}{\operatorname{Tor}}

\newcommand{\ov}{\overline}

\newcommand{\ord}{\operatorname{ord}}

\newcommand{\discsym}{\operatorname{disc-sym}}

\newcommand{\PD}{\operatorname{PD}}

\title[Finite abelian group actions on WLS manifolds]
{Finite abelian group actions on weakly Lefschetz cohomologically symplectic manifolds}
\author{Ignasi Mundet i Riera}
\address{Facultat de Matem\`atiques i Inform\`atica\\
Universitat de Barcelona\\
Gran Via de les Corts Catalanes 585\\
08007 Barcelona, Spain}
\address{Centre de Recerca Matem\`atica \\
Campus de Bellaterra, Edifici C \\
08193 Barcelona, Spain}
\email{ignasi.mundet@ub.edu}

\date{\today}


\thanks{This research was partially supported by the Spanish Ministeri de Ci\`encia i Innovaci\'o,
through the grant PID2023-147642NB-I00, and by the Spanish State Research Agency, through the
Severo Ochoa and Mar\'{\i}a de Maeztu Program for Centers and Units of Excellence in
R\&D (CEX2020-001084-M)}

\begin{document}

\maketitle

\begin{abstract}
We study finite abelian group actions on {\it weakly Lefschetz cohomologically symplectic} (WLS) manifolds, a collection of
manifolds that includes all compact connected Kaehler manifolds.

We prove that
for any WLS manifold $X$ there exists a number $C$ such that, for any integer $m\geq C$,
if $(\ZZ/m)^k$ acts freely on $X$, then $\sum_j b_j(X;\QQ)\geq 2^k$.

We also prove a structure theorem for effective actions on WLS manifolds of $(\ZZ/p)^r$, where $p$ is a big enough prime,
analogous to some results for tori of Lupton and Oprea \cite{LO},
and we find bounds on the discrete degree of symmetry of WLS manifolds.

Our technique, which may be of independent interest, is based on studying the cohomology of abelian covers of WLS manifolds
$X$ associated to certain maps $\pi:X\to T^k$. We prove that, in the presence of actions of arbitrarily big finite abelian groups,
some of these abelian covers have finitely generated cohomology, and the spectral sequence associated to $\pi$ degenerates at the second page over the rationals.
\end{abstract}

\section{Introduction}

\newcommand{\discrk}{\operatorname{disc-rk}}

Let us say that a closed,
connected and orientable $n$-dimensional topological manifold $X$ is {\it weakly Lefschetz
cohomologically symplectic} (WLS, for short) if there exists a class $\Omega\in H^2(X;\RR)$ such that:
\begin{enumerate}
\item[(W1)] there exist classes $\alpha_1,\dots,\alpha_r\in H^1(X;\RR)$ such that
$\alpha_1\smile\dots\smile\alpha_r\smile\Omega^k\neq 0$, where $r+2k=n$ and $k\geq 0$,
\item[(W2)] for every nonzero $\alpha\in H^1(X;\RR)$ there exist classes
$\beta_1,\dots,\beta_s\in H^1(X;\RR)$ such that
$\alpha\smile\beta_1\smile\dots\smile\beta_s\smile\Omega^t\neq 0$, where $1+s+2t=n$
and $s\geq 0$.
\end{enumerate}
We call any such $\Omega$ a WLS class of $X$. If $\Omega=0$, condition (W1) above
implies that $k=0$, so that $X$ is rationally hypertoral in the sense of \cite{Mu2021}. In this case,
all the results in this paper follow, without using (W2), from the results in \cite{Mu2021} (except for
Theorem \ref{thm:main-big-prime-p} below, which also needs \cite[Theorem 1.6]{Mu2022}).
So the reader can assume from now on that all WLS classes are implicitly nonzero.

By Poincar\'e duality, $\Omega\in H^2(X;\RR)$ is a WLS class if and only if
the image of the cup product map
$\Lambda^*H^1(X;\RR)\otimes_{\RR}\RR[\Omega]\to H^*(X;\RR)$
contains $H^{n-1}(X;\RR)\oplus H^n(X;\RR)$.

The aim of this paper is to study effective actions\footnote{All group actions in this paper
are implicitly assumed to be continuous.} of finite abelian groups on WLS manifolds.

The collection of WLS manifolds contains all $c$-symplectic manifolds of Lefschetz type,
as defined by Lupton and Oprea \cite{LO}: these are the closed, connected, $2m$-dimensional manifolds $X$
admitting a class $\Omega\in H^2(X;\RR)$ such that $\Omega^{m}\neq 0$,
and such that cup product with $\Omega^{m-1}$ gives an isomorphism from $H^1(X;\RR)$ to $H^{2m-1}(X;\RR)$.
Indeed, any such $\Omega$ is a WLS class.
By the Lefschetz decomposition theorem (see. e.g. \cite[Chap. V, Corollary 4.13]{Wells}),
any compact connected Kaehler manifold is $c$-symplectic of Lefschetz type, with $\Omega$
equal to the cohomology class of the symplectic form. Hence, any compact connected Kaehler manifold
is a WLS manifold.

WLS are not necessarily even dimensional:
any compact co-Kaehler manifold is WLS, see \cite[Theorem 2.1, Lemma 2.3]{BLO}.

The Cartesian product of two WLS manifolds is again WLS. Indeed, if $X_1,X_2$ support WLS classes, say $\Omega_1,\Omega_2$
respectively, then $\Omega_1+\Omega_2\in H^2(X_1\times X_2;\RR)$ is a WLS class (we omit the pullback maps induced by the projections $X_1\times X_2\to X_i$). Indeed, suppose that
$r_i,t_i$ are nonnegative integers such that $r_i+2t_i=\dim X_i$, and that $\alpha_{1,i},\dots,\alpha_{r_i,i}\in H^1(X_i;\RR)$.
Then, elementary dimension considerations imply that
\begin{multline*}
\left(\alpha_{1,1}\smile\dots\smile\alpha_{r_1,1}\smile\Omega_1^{t_1}\right)\smile
\left(\alpha_{1,2}\smile\dots\smile\alpha_{r_2,2}\smile\Omega_2^{t_2}\right)= \\ =\frac{t_1!t_2!}{(t_1+t_2)!}\cdot
\alpha_{1,1}\smile\dots\smile\alpha_{r_1,1}\smile
\alpha_{1,2}\smile\dots\smile\alpha_{r_2,2}\smile(\Omega_1+\Omega_2)^{t_1+t_2}
\end{multline*}
(again, pullbacks omitted). This formula can be applied to prove that
$\Omega_1+\Omega_2$ satisfies both properties (W1) and (W2).

Throughout this paper, $H^*(\cdot)$ denotes integral cohomology.
Denote by $T^k$ the $k$-dimensional torus $(S^1)^k$.

\subsection{A weak version of Carlsson's toral rank conjecture for WLS manifolds}

The following conjecture was proposed by G. Carlson in \cite[Conjecture I.3]{Carlsson1986} and
in \cite{Carlsson1987}.

\begin{conjecture}
\label{conj:CTRC}
Let $p$ be a prime and let $G=(\ZZ/p)^k$. If $X$ is a finite free $G$-complex, then
$\sum_i\dim_{\ZZ/p}H_i(X,\ZZ/p)\geq 2^k$.
\end{conjecture}

In this paper we refer to this conjecture as Carlsson's toral rank conjecture. The analogous conjecture,
in which $(\ZZ/p)^k$ is replaced by $T^k$, "free" is replaced by "with finite stabilizers",
and $\dim_{\ZZ/p}H_i(X,\ZZ/p)$ is replaced by $\dim_{\QQ}H_i(X,\QQ)$, was also proposed by
Carlsson in \cite[Conjecture I.3]{Carlsson1986} and by S. Halperin (the conjecture appears
at least in two papers by Halperin, with different hypothesis: in \cite[\S 8]{Halperin1987},
for smooth actions on compact differentiable manifolds, and in \cite[Problem 1.4]{Halperin1985},
for actions on {\it reasonable} simply connected spaces). We refer to it as Halperin's toral
rank conjecture. (See also \cite[Conjecture 4.4.1]{AP}.)

A number of partial results on Carlsson's toral rank conjecture have been proved in the past,
but, except for some general results on low dimensional manifolds, all of them refer
exclusively to products of spheres or manifolds closely related to them.
For actions on products of spheres, with restrictions on the dimensions, there are partial
results by A. Adem, W. Browder, G. Carlsson, D.H. Gottlieb, O.B. Okutan and E. Yal\c{c}{\i}n, see
\cite{AB,Carlsson1980,Carlsson1982,Carlsson1987,Gottlieb,OY}.
For actions of $(\ZZ/2)^k$ on products of Dold manifolds (which are quotients of $S^m\times\CP^n$ by a suitable involution), it has been recently proved by P. Dey in \cite{Dey}. For products of
equidimensional lens spaces, it was proved by Yal\c{c}{\i}n in \cite{Yalcin}.
For products of real projective spaces, see the recent paper by Yal\c{c}{\i}n \cite{Yalcin2025}.
For free actions of $(\ZZ/2)^k$ whose orbit space is a small cover, it was proved by L. Yu
in \cite{Yu1}, and in \cite{Yu2} for actions of $(\ZZ/2)^k$ on manifolds of dimension at most $3$.

B. Hanke proved in \cite{Hanke} a weak version of Carlsson's conjecture for an arbitrary product of spheres
$X=S^{n_1}\times\dots\times S^{n_r}$: namely, for every prime $p>3\dim X$, if there is a free action
of $(\ZZ/p)^k$ on $X$ we have $k\leq r_o$, where $r_o$ denotes the cardinal of $\{i\mid n_i\text{ is odd}\}$.

Our first result implies a weak version of Carlsson's toral rank conjecture for WLS manifolds.
To be precise, what we prove is weaker than the original conjecture, because it does not apply
to actions of $(\ZZ/p)^k$ for small primes $p$,
but it is more general, in the sense that $(\ZZ/p)^k$ is replaced by $(\ZZ/p^e)^k$,
with $p^e$ big enough, and that the action may have nontrivial isotropy groups, provided
they are small enough.

\begin{theorem}
\label{thm:toral-rank-Carlsson}
Let $X$ be a WLS manifold. For each number $C$ there is a number
$C'$, depending on $C$ and $X$, such that, for any prime power $p^e\geq C'$ and any action
of $(\ZZ/p^e)^k$ on $X$ all of whose isotropy groups have at most $C$ elements,
there is a continuous map $X\to T^k$ inducing an injective morphism
$\phi^*:H^*(T^k)\to H^*(X)$.
\end{theorem}

Leaving aside the possible presence of nontrivial isotropy groups this result
has a similar flavour to that of Hanke in \cite{Hanke}. However, our technique is very different.
While Hanke uses tame homotopy theory, our result is based on studying the monodromy action
on the cohomology of certain appropriately chosen abelian covers of $X$
(see Subsection \ref{ss:intro-rotation-morphism} below).

Theorem \ref{thm:toral-rank-Carlsson} immediately implies the following:

\begin{corollary}
\label{cor:toral-rank-Carlsson}
Let $X$ be a WLS manifold. There is a number $C$, depending only on $X$, such that for any
integer $m\geq C$ and any free action of $(\ZZ/m)^k$ on $X$ we have
$\sum_j b_j(X;\QQ)\geq 2^k$.
\end{corollary}

The conclusion in Theorem \ref{thm:toral-rank-Carlsson} is stronger than that of
the previous corollary, because it asserts that the bound $\sum_j b_j(X;\QQ)\geq 2^k$
has a geometric origin, namely a continuous map $X\to T^k$ inducing an injection in
cohomology. It is an intriguing question to understand for which manifolds $X$ the existence
of a free action of $(\ZZ/p)^k$, with $p$ big enough, implies the existence of
a cohomologically injective continuous map $X\to T^k$. This is certainly not true for all manifolds
$X$. For example, $(S^3)^k$ supports a free action of
$(\ZZ/p)^k$ for every $p$, but since $(S^3)^k$ is simply connected there cannot be any
cohomologically injective map $(S^3)^k\to T^k$.

Corollary \ref{cor:toral-rank-Carlsson} implies that WLS manifolds
satisfy Halperin's toral tank conjecture:
if $T^k$ acts on a WLS with finite stabilizers, restricting the action to
the $p$-torsion of $T^k$, for big enough $p$, we get a free action of
$(\ZZ/p)^k$ on $X$. However, for the more restricted family of $c$-symplectic
manifolds of Lefschetz type, Halperin's toral rank conjecture can be proved
much more directly, see the comments after Theorem \ref{thm:LO} below.


\subsection{Actions of elementary abelian $p$-groups, with $p$ big}
G. Lupton and J. Oprea proved in \cite{LO} several results on continuous torus actions on $c$-symplectic
manifolds of Lefschetz type, generalizing previously known results on symplectic torus actions on
symplectic manifolds as in \cite{Ginzburg}.
The following result combines \cite[Theorem 5.2]{LO},
\cite[Corollary 5.6]{LO}, and \cite[Corollary 5.15]{LO}.

\begin{theorem}
\label{thm:LO}
Let $X$ be a $c$-symplectic manifold of Lefschetz type. For any effective action
of a torus $T^r$ on $X$ there is a splitting $T^r\simeq T^k\times T^{r-k}$ such that:
\begin{enumerate}
\item $T^k$ acts on $X$ with finite stabilizers,
and there is a ring isomorphism $H^*(X;\QQ)\simeq H^*(X/T^k;\QQ)\otimes H^*(T^k;\QQ)$,
\item the action of $T^{r-k}$ on $X$ has nonempty fixed point set.
\end{enumerate}
\end{theorem}

The first statement above implies that if a torus $T^k$ acts effectively
and with finite stabilizers on a $c$-symplectic manifold $X$ of Lefschetz type then
$\sum_j b_j(X;\QQ)\geq 2^k$, which gives an affirmative answer, for such
manifolds, to Halperin's toral rank conjecture.
This was also independently proved by C. Allday, see Remark 1.6.6 in \cite{AMMS}.
Halperin's conjecture had earlier been proved for spaces with the rational
cohomology of a compact Kaehler manifold by C. Allday and V. Puppe \cite{AP0}.
For co-Kaehler manifolds, it has been more recently proved by
G. Bazzonni, G. Lupton and J. Oprea \cite{BLO}.
L. Zoller wrote recently a very informative and up to date survey on Halperin's
toral rank conjecture, see \cite{Zoller}.

Lupton and Oprea prove in \cite{LO} that the isomorphism in the first statement
of Theorem \ref{thm:LO} has a geometric origin: $X$ is in fact $T^k$-equivariantly homeomorphic to
$Y\times_FT^k$, where $F$ is a finite subgroup of $T^k$, and where $T^k$ acts
on the second factor of $Y\times_FT^k$ by translations and trivially on the first
factor. A similar statement has been proved for holomorphic actions on Kaehler manifolds
(which are particular cases of $c$-symplectic manifolds of Lefschetz type)
by J. Amor\'os, M. Manjar\'{\i}n and M. Nicolau, see \cite[Corollary 7.12]{AMN}.


Our second result is the following analogue of Theorem \ref{thm:LO} for actions of elementary $p$-groups on WLS manifolds
(see \cite[Chap I]{Bo} for the definition of $\ZZ/p$-cohomology manifold and
of the dimension of a $\ZZ/p$-cohomology manifold):

\begin{theorem}
\label{thm:main-big-prime-p}
Let $X$ be a WLS manifold. There is a number $C$, depending only on $X$, with the following property.
Let $p\geq C$ be a prime number, and suppose that $G:=(\ZZ/p)^r$ acts effectively on $X$. Then there is a
splitting $G\simeq G_0\times G_1$, a monomorphism $G_1\hookrightarrow T^k$,
and a map $\xi:X\to T^k$,
such that:
\begin{enumerate}
\item $\xi$ is $G_1$-equivariant (so $G_1$ acts freely on $X$),
\item there is a finite covering $\pi:X'\to X$, a $\ZZ^k$-covering $X''\to X$,
and an isomorphism of graded vector spaces
$\sigma:H^*(X';\QQ)\to H^*(X'';\QQ)\otimes H^*(T^k;\QQ)$
such that the composition of the maps
$$H^*(T^k;\QQ)\stackrel{\xi^*}{\longrightarrow} H^*(X;\QQ)\stackrel{\pi^*}{\longrightarrow}
H^*(X';\QQ) \stackrel{\sigma}{\longrightarrow}H^*(X'';\QQ)\otimes H^*(T^k;\QQ)$$
sends $H^*(T^k;\QQ)$ isomorphically onto $1\otimes H^*(T^k;\QQ)$,
\item the action of $G_0$ on $X$ has nonempty fixed point set, and
each connected component of $X^{G_0}$ is a $\ZZ/p$-cohomology manifold
of dimension $\geq k$.
\end{enumerate}
\end{theorem}

In the previous theorem we may write $G_1\simeq (\ZZ/p)^s$ and $G_0\simeq (\ZZ/p)^{r-s}$, and the existence of a monomorphism
$G_1\hookrightarrow T^k$ implies that $s\leq k$. 
The manifold $X''$ is a model of the homotopy fiber of the map $\xi:X\to T^k$,
so the isomorphism $H^*(X';\QQ)\to H^*(X'';\QQ)\otimes H^*(T^k;\QQ)$ may be seen as a cohomological
analogue of the homeomorphism $X\simeq T^k\times_FY$ obtained by Lupton an Oprea in the context of actions
of $T^k$ on Lefschetz $c$-manifolds (or of the similar decomposition obtained by Amor\'os, Manjar\'{\i}n and Nicolau \cite{AMN}
for holomorphic actions on Kaehler manifolds).


\subsection{Discrete degree of symmetry}
Our techniques also lead to results involving general finite abelian group actions.
To state them we use the notion of discrete degree of symmetry \cite{Mu2021}, whose definition we now
recall. Let $X$ be a closed topological manifold. Define $\mu(X)$ to be the set of natural numbers $m$ such that
there exists a sequence of integers $r_i\to\infty$
and continuous effective actions, for each $i$, of $(\ZZ/r_i)^m$ on $X$.
By a theorem of L.N. Mann and J. C. Su \cite[Theorem 2.5]{MannSu}, the set $\mu(X)$ is finite.
The discrete degree of symmetry of $X$ is defined as
$$\discsym(X):=\max\{0\}\cup\mu(X).$$
It is possible to rephrase the definition into a statement involving effective actions of arbitrary
finite abelian groups, see \cite[Lemma 2.7]{Mu2021}.

We need yet another definition. Let $\tau(X)$ be the biggest $r$ for which there exist classes $\alpha_1,\dots,\alpha_r\in H^1(X;\RR)$
such that $\alpha_1\smile\dots\smile\alpha_r\neq 0$. If $r=\tau(X)$ then there exists a continuous map
$\phi:X\to T^r$ such that $\phi^*H^r(T^r;\RR)\subset H^r(X;\RR)$ is nonzero. By Poincar\'e duality, this
implies that $\phi^*:H^*(T^r;\RR)\to H^*(X;\RR)$ is injective, so $b_j(X;\QQ)\geq\left(r\atop j\right)$
for every $0\leq j\leq r$.

This is our last main result:

\begin{theorem}
\label{thm:main-disc-deg-sym}
Let $X$ be a WLS $n$-dimensional manifold. We have
$$\discsym (X)\leq \frac{n+\tau(X)}{2}.$$
If $\discsym (X)=n$ then there is an isomorphism of rings
$H^*(X)\simeq H^*(T^{n})$. Finally, if $\pi_1(X)$ is virtually solvable
and $\discsym (X)=n$, then $X$ is homeomorphic to $T^{n}$.
\end{theorem}

It was asked in \cite[Question 1.1]{Mu2021} whether the inequality
$\discsym(X)\leq\dim X$ holds for every closed connected manifold $X$,
and whether the equality $\discsym(X)=\dim X$ is only satisfied by tori.
The previous theorem gives an affirmative answer to the first question for
WLS manifolds, since we clearly have $\tau(X)\leq\dim X$,
and it gives strong evidence for the second one.
This extends \cite[Theorem 1.3]{Mu2021}, which answers the question for
rationally hypertoral manifods.
Other partial positive answers to this question are given in
\cite[Theorems 1.3, 1.4]{Mu2022}, for manifolds with nonzero Euler characteristic
and for integral homology spheres.

Since compact connected Kaehler manifold are WLS,
Theorem \ref{thm:main-disc-deg-sym} generalizes the final part of \cite[Theorem 1.9]{Mu2021}
from holomorphic to continuous actions.

F. Catanese proved in \cite[Theorem 77]{Catanese} that if a compact Kaehler manifold $X$
has the integral cohomology of a torus, then $X$ is biholomorphic to a
complex torus. This fact, combined with Theorem \ref{thm:main-disc-deg-sym}, implies the following.

\begin{corollary}
\label{cor:kaehler-torus}
Let $X$ be a closed connected Kaehler manifold. Suppose that, for some
natural number $m$ and for some sequence of natural numbers $r_i\to\infty$, there is an effective
and continuous action of $(\ZZ/r_i)^m$ on $X$. Then $m\leq\dim_{\RR}X$, and if
$m=\dim_{\RR}X$ then $X$ is biholomorphic to a complex torus.
\end{corollary}

See \cite[Theorem 1.5]{Gol} for a similar result, by A. Golota, for $X$ projective and where continuous actions
are replaced by actions through birational transformations.

\subsection{Group actions, maps to tori, and abelian covers}
\label{ss:intro-rotation-morphism}
The main technical contribution of this paper, on which all the previous results are based,
relates the existence of effective actions of arbitrarily large finite abelian groups
on a manifold $X$
to a finiteness property for the homology of certain abelian covers of $X$,
and to the degeneracy over $\QQ$ of the spectral sequence for the classifying
maps these covers. This subsection is devoted to explaining and stating this technical result.

In the first next paragraphs we introduce some notation that will be used throughout this paper.

\subsubsection{Abelian covers and maps to tori}
\label{sss:abelian-covers-tori}
We begin with a general construction to obtain abelian covers.
Let $V$ be a finite dimensional real vector space, and let $\Lambda<V$ be a lattice.
For any $v\in V$ we denote by $[v]\in V/\Lambda$ the image of $v$ under the projection
map $V\to V/\Lambda$. Given a topological space $Z$ and a continuous map $\psi:Z\to V/\Lambda$
we define
$$Z_{\psi}=\{(z,v)\in Z\times V\mid \psi(z)=[v]\}.$$
This is the total space of the pullback of the covering $V\to V/\Lambda$ through the map $\psi$.
The action of
$\Lambda$ on $V$, together with the trivial action on $Z$, induces a diagonal action on
$Z\times V$ that preserves $Z_{\psi}$. With respect to this action, the obvious projection $Z_\psi\to Z$
has a structure of principal $\Lambda$-bundle. Up to homeomorphisms lifting the identity on $Z$,
the bundle $Z_{\psi}\to Z$
only depends  on the homotopy class of the map $\psi$
(see Lemma \ref{lemma:homotopic-bundles-isomorphic}).
Note that, if $Z$ is path connected, then $Z_{\psi}$ is path connected if and only if $\psi_*:H_1(Z)\to H_1(V/\Lambda)$ is surjective.

\subsubsection{Finitely generated free $\ZZ$-modules and tori}
\label{sss:modules-and-tori}
Throughout this paper we identify $S^1$ with $\RR/\ZZ$. Let $\pi:\RR\to\RR/\ZZ$ be the quotient
map. Accordingly, we use additive notation for the group operation on $S^1$.
Denote by $\theta\in H^1(S^1)$ the class represented in de Rham cohomology by a form
in $\Omega^1(S^1)=\Omega^1(\RR/\ZZ)$ whose pullback to $\RR$ via $\pi$ is equal to $dx$,
where $x$ is the canonical coordinate in $\RR$. Then $\theta$ is a generator of $H^1(X)$.

For any finitely generated free $\ZZ$-module $A$ we denote
$$T_A:=\Hom(A,S^1).$$
With its natural topology, induced from the standard topology of $S^1$, $T_A$ is a
torus of dimension $\dim A\otimes_{\ZZ}\QQ$. Any morphism of finitely generated
free $\ZZ$-modules $f:A'\to A$ induces by precomposition a continuous map $T_f:T_A\to T_{A'}$.
In particular, if $A'\leq A$ is a submodule, then we have a restriction morphism
$T_A\to T_{A'}$.

For any $a\in A$, denote by $e_a:T_A\to S^1$ the evaluation at $a$.
The map
$$\epsilon_A:A\to H^1(T_A),$$
sending each $a\in A$ to $e_a^*\theta$, is an isomorphism of groups.
In particular, if $A$ is cyclic generated by $a$ then $\epsilon_A(a)\in H^1(T_A)$ is a generator.
The isomorphism $\epsilon_A$ is natural: if $f:A'\to A$ is a morphism
of finitely generated free $\ZZ$-modules, then 
\begin{equation}
\label{eq:naturality-epsilon-A}
\epsilon_A\circ f=T_f^*\circ\epsilon_{A'}:A'\to H^1(T_A).
\end{equation}

\subsubsection{The canonical element in $[X,T_{H^1(X)}]$}
Let $X$ be a closed connected topological manifold.
The group $H_1(X)$ is finitely generated,
and consequently $H:=H^1(X)\simeq\Hom(H_1(X),\ZZ)$ is a finitely generated free abelian group.

%

\begin{lemma}
\label{lemma:canonical-class}
There exist a map
$\psi:X\to T_H$, unique up to homotopy,
such that $\psi^*\circ\epsilon_H$ is the identity in $H$,
where $\psi^*:H^1(T_H)\to H^1(X)=H$ is the map induced by $\psi$.
\end{lemma}
\begin{proof}
Let $h_1,\dots,h_s$ be a basis of $H$.  Since $S^1$ is an Eilenberg--Maclane space $K(\ZZ,1)$,
the map $[X,S^1]\to H^1(X)$
sending the class of $\phi:X\to S^1$ to $\phi^*\theta\in H^1(X)$ is a bijection,
so we may pick maps $\psi_1,\dots,\psi_s:X\to S^1$ satisfying $\psi_i^*\theta=h_i$.
The map $\psi:X\to T_H$ defined by $\psi(x)(\sum\lambda_ih_i)=\sum\lambda_i\psi_i(x)$ satisfies
$\psi^*\circ\epsilon_H=\Id_H$. This proves existence. For uniqueness up to homotopy, let
$\la h_i\ra\leq H$ be the subgroup generated by $h_i$. The restriction maps $\rho_i:T_H\to T_{\la h_i\ra}$
combine to give a homeomorphism $\prod_i\rho_i:T_H\to\prod_i T_{\la h_i\ra}$. Suppose that $\psi:X\to T_H$ satisfies
$\psi^*\circ\epsilon_H=\Id_H$. Write $(\prod_i\rho_i)\circ\psi=(\psi_1,\dots,\psi_s)$.
Then by (\ref{eq:naturality-epsilon-A}) we have
$\psi_i^*(\epsilon_{\la h_i\ra}(h_i))=(\rho_i\circ\psi)^*(\epsilon_{\la h_i\ra}(h_i))=\psi^*(\epsilon_H(h_i))=h_i$. Since
$T_{\la h_i\ra}\simeq S^1=K(\ZZ,1)$, this implies that each $\psi_i$, and hence also $\psi$, is uniquely prescribed up to homotopy.
\end{proof}

We call the homotopy class represented by the maps in the previous lemma the
{\it canonical element} of $[X,T_H]$.

\subsubsection{Subgroups of $H$ and abelian covers of $X$}
Let $K$ be a subgroup of $H$. Define
$V_K:=\Hom(K,\RR).$
The real vector space structure on the target $\RR$
induces on $V_K$ a structure of real vector space, and the group $K^*=\Hom(K,\ZZ)$ sits as
a lattice in $V_K$.
The map
$$\pi_K:V_K\to T_K,$$
sending
any morphism $K\to\RR$ to its composition with the projection $\RR\to\RR/\ZZ=S^1$,
is a universal cover map, and it gives a homeomorphism $V_K/K^*\cong T_K$.

Let $\psi:X\to T_H$ represent the canonical element in $[X,T_H]$.
Denote by $r_K:T_H\to T_K$ the
restriction morphism. Define
$\psi_K:=r_K\circ \psi:X\to T_K$.
Applying the construction in \S\ref{sss:abelian-covers-tori} to $V=V_K$, $\Lambda=K^*$,
and the map $\psi_K$, we define
$$X_K=X_{\psi_K}.$$
In particular, the natural projection $X_K\to X$ has a structure of principal $K^*$-bundle
(see Subsection \ref{ss:abelian-covers} for more details).

Note that $X_K$ is connected if and only if $H/K$ is torsion free. To see this, note that
in terms of the identifications
$H_1(T_H)\simeq H^*$ and $H_1(T_K)\simeq K^*$ obtained by dualizing
$\epsilon_H$ and $\epsilon_K$, the map ${r_K}_*:H_1(T_H)\to H_1(T_K)$ corresponds
to the dual $\iota^*:H^*\to K^*$ of the inclusion $\iota:K\hookrightarrow H$.
If $H/K$ has nontrivial torsion, then $\iota^*$ fails to be surjective, and
hence ${\psi_K}_*={r_K}_*\circ \psi_*:H_1(X)\to H_1(T_K)$ can neither be surjective.
If $H/K$ is torsion free then $\iota^*$ is indeed surjective,
and hence so is $(\psi_K)_*$ (indeed, $\psi_*:H_1(X)\to H_1(T_H)$ is surjective,
since $\psi$ represents the canonical class in $[X,T_H]$), so in this case $X_K$ is connected.
If $H/K$ is torsion free then there exists a subgroup $K'\leq H$ satisfying $H=K\oplus K'$,
and $X_K$ coincides with
the covering of $X$ corresponding to the subgroup of $\pi_1(X)$ equal to the preimage
of $K'$ under the abelianization map $\pi_1(X)\to H$.

\subsubsection{Definition of the rotation morphism}

Let us say that an action of a group $G$ on $X$ is {\em $H^1$-trivial} (resp. $H^*$-trivial)
if the induced action of $G$ on $H^1(X)$ (resp. $H^*(X)$) is trivial.

Suppose that a finite group $G$ acts $H^1$-trivially on $X$.
Denote by $\{E_r^{p,q}\}$
the entries of the Serre spectral sequence for the Borel fibration $X_G\to BG$.
Since $G$ acts trivially on $H^1(X)$, we have $E_2^{0,1}=H^0(BG;H^1(X))\simeq H^1(X)$.
Define $\xi:H^1(X)\to\Hom(G,S^1)$ to be the composition of the following morphisms:
$$H^1(X)\simeq E_2^{0,1}\stackrel{d_2^{0,1}}{\longrightarrow}
E_2^{2,0}=H^2(BG;H^0(X))\simeq H^2(G)\stackrel{\mu}{\longrightarrow}\Hom(G,S^1),$$
where $\mu$ is the composition of
\begin{enumerate}
\item the inverse of the connecting morphism
$H^1(G;S^1)\to H^2(G)$ in the long exact sequence arising from the
exact sequence of coefficients $0\to\ZZ\to\RR\to \RR/\ZZ=S^1\to 0$ (recall that
$H^k(G;\RR)=0$ for every $k>0$, because $G$ is finite), and
\item the natural isomorphism $H^1(G;S^1)\simeq\Hom(G,S^1)$,
\end{enumerate}
and $\mu$ is hence an isomorphism.

Denote as before $H=H^1(X)$.
The {\it rotation morphism} (for the action of $G$ on $X$) is the morphism of groups
$$\rho:G\to T_H$$
defined as
$\rho(g)(h)=\xi(h)(g)$ for every $g\in G$ and $h\in H$
(see \cite{Mu2010,MundetSaez,Mu2021} for equivalent definitions of $\rho$).

\subsubsection{}
The following two results are the main tools of this paper.

\begin{theorem}
\label{thm:main-rotation-morphism}
Let $X$ be a closed connected topological manifold.
For any number $C_1$ there is a number $C_2=C_2(X,C_1)$ such that:
for any prime $p$, and for any $H^1$-trivial action of a finite $p$-group $G$ on $X$,
whose rotation morphism $\rho:G\to T_{H^1(X)}$ satisfies
$|\Ker\rho|\leq C_1$, there is a subgroup $K\leq H^1(X)$ satisfying the following properties:
\begin{enumerate}
\item the kernel of the restriction map $r_K:\rho(G)\to T_K$ has at most $C_2$ elements,
\item $H^*(X_K)$ is finitely generated as a $\ZZ$-module,
\item there is a subgroup of finite index $K^*_0\leq K^*$ whose induced action
on $H^*(X_K)$ is trivial, and the Serre spectral sequence over the rationals of the fibration
$$X_K\times_{K^*_0} V_K\to V_K/K^*_0$$
degenerates in the second page, and
\item $\psi_K^*:H^*(T_K)\to H^*(X)$ is injective.
\end{enumerate}
\end{theorem}

The previous theorem is valid for arbitrary closed topological manifolds.
What makes it particularly useful in the case of WLS manifolds is that,
on these manifolds, there is a relation between the nonvanishing of the
kernel of the rotation morphism and the existence of points with big stabilizer subgroup. This
is the content of the following result.

\begin{theorem}
\label{thm:fixed-points-kernel-rotation-morphism}
Suppose that $X$ is an $n$-dimensional WLS manifold. There is a constant $C_3$ such that,
for every prime $p$ and for every effective and $H^*$-trivial
action of a (finite) cyclic $p$-group $G$ on $X$ with trivial rotation morphism,
there exists some $x\in X$ whose stabilizer $G_x$ satisfies
$|G|\leq C_3\cdot|G_x|^{n/2}.$
\end{theorem}

The group actions in Theorems
\ref{thm:toral-rank-Carlsson}, \ref{thm:main-big-prime-p} and \ref{thm:main-disc-deg-sym}
are not assumed to induce a trivial action on $H^*(X)$, and not even on $H^1(X)$, unlike
the actions considered in
Theorems \ref{thm:main-rotation-morphism} and \ref{thm:fixed-points-kernel-rotation-morphism}.
Despite this fact, we may use Theorems \ref{thm:main-rotation-morphism} and
\ref{thm:fixed-points-kernel-rotation-morphism} to prove
Theorems
\ref{thm:toral-rank-Carlsson}, \ref{thm:main-big-prime-p} and \ref{thm:main-disc-deg-sym},
thanks to the following
result, which follows from a well nown lemma of Minkowski (see
the proof of Lemma \ref{lemma:Minkowski-X-K} below).

\begin{lemma}
\label{lemma:Minkowski}
For any closed manifold $X$ there exists a natural number $C_M$, depending only on $X$,
such that for any action on $X$ of a finite group $G$ there is a
subgroup $G'\leq G$ whose action on $H^*(X)$ is trivial, and which satisfies
$[G:G']\leq C_M$.
\end{lemma}

\subsection{Contents}
Section \ref{s:rotation-morphism} gives an alternative, more geometric, characterization of the rotation morphism. In Section \ref{s:proof-thm:main-rotation-morphism} we prove Theorem \ref{thm:main-rotation-morphism}. Section \ref{s:proof-thm:trivial-module-structure}
contains the proof of a technical result that was used in proving Theorem
\ref{thm:main-rotation-morphism}. Theorem \ref{thm:fixed-points-kernel-rotation-morphism} is proved in Section \ref{s:integral-cohomology}. Section \ref{s:proofs-main-thms} contains the
proofs of the results stated in the introduction.

\subsection{More notation}
If $M$ is a $\ZZ$ module, we denote by $\Tor M$ the torsion of $M$.
We use $\leq$ to denote group inclusion. We identify
$T^k$ with $\RR^k/\ZZ^k$. We use additive notation
for the group operation on abelian groups.

\subsection{Acknowledgements}
I am very pleased to thank G. Lupton and J. Oprea for several useful comments.
Many thanks also to the referee for a very detailed and helpful report.

\section{Equivariant maps to tori and the rotation morphism}

\label{s:rotation-morphism}

\begin{theorem}
\label{thm:rotation-morphism}
Let $X$ be a closed connected topological manifold, and let $H:=H^1(X)$.
For any $H^1$-trivial action of a finite group $G$ on $X$, there is a continuous
map $\phi:X\to T_H$ satisfying:
\begin{enumerate}
\item[(R1)] $\phi$ represents the canonical element in $[X,T_H]$,
\item[(R2)] 
$\phi(g\cdot x)=\rho(g)+\phi(x)$ for each $x\in X$ and $g\in G$,
where $\rho:G\to T_H$ is the rotation morphism.
\end{enumerate}
\end{theorem}
\begin{proof}
We use the same ideas as in the proof of \cite[Lemma 4.2]{Mu2021}.
Let $\psi:X\to T_H$ be a map representing the canonical element in $[X,T_H]$,
as given by Lemma \ref{lemma:canonical-class}.
Suppose that a finite group $G$ acts $H^1$-trivially on $X$.
For each $g\in G$, let $g^*\psi:X\to T_H$
denote the map defined by $(g^*\psi)(x):=\psi(g\cdot x)$.
Since $G$ acts $H^1$-trivially on $X$,
the maps $g^*\psi$ and $\psi$ are homotopic
for each $g\in G$, so $\xi:=-|G|\cdot\psi+\sum_{g\in G}g^*\psi$ represents the
trivial element in $[X,T_H]$. Hence, there exists a map $\xi':X\to V_H=\Hom(H,\RR)$
satisfying $\xi=\pi_H\circ\xi'$, where $\pi_H:V_H\to T_H$ is the projection, which is the
universal covering space of $T_H$.
Let $\chi:=\pi_H(\xi'/|G|)$ and let $\phi:=\chi+\psi$. The map $\chi$ represents the trivial element in $[X,T_H]$
(because it has a lift $X\to V_H$, namely $\xi'/|G|$), so $\phi$ and $\psi$ are
homotopic, which proves (R1).
We also have
$$|G|\cdot\chi=\xi=-|G|\cdot\psi+\sum_{g\in G}g^*\psi,$$
which implies rearranging that
$|G|\cdot\phi=|G|\cdot(\chi+\psi)=\sum_{g\in G}g^*\psi$.
The latter implies that $|G|\cdot (g^*\phi)=g^*(|G|\cdot \phi)=|G|\cdot\phi$
for each $g\in G$,
from which it follows that there exists some $\eta_g\in T_H$ such that
$g^*\phi-\phi=\eta_g$. Clearly, the map $\eta:G\to T_H$, $\eta(g):=\eta_g$ is a morphism of groups. Using this fact, property (R2) follows from the next lemma.
\end{proof}

\begin{lemma}
Let a finite group $G$ act $H^1$-trivially on a closed connected topological manifold
$X$. Let $H:=H^1(X)$.
Suppose that a continuous $\phi:X\to T_H$ represents the canonical element in $[X,T_H]$
and that $\eta:G\to T_H$ is
a morphism of groups such that $\phi(g\cdot x)=\eta(g)+\phi(x)$ for every
$g\in G$ and $x\in X$. Then $\eta$ is equal to the rotation morphism of
the action of $G$ on $X$.
\end{lemma}
\begin{proof}
Let $\xi:H^1(X)\to\Hom(G,S^1)$ be the morphism defined in
Subsection \ref{ss:intro-rotation-morphism}.
Since $\phi(g\cdot x)=\eta(g)+\phi(x)$ for every
$g\in G$ and $x\in X$,
by the naturality of the Serre spectral sequence and the Borel construction, we have
$\xi\circ \phi^*=\xi_H$,
where 
$\xi_H$ is the composition
$$H^1(T_H)\simeq E_2^{0,1}\stackrel{d_2^{0,1}}{\longrightarrow}
E_2^{2,0}=H^2(BG;H^0(T_H))\simeq H^2(G)\stackrel{\mu}{\longrightarrow}\Hom(G,S^1),$$
i.e. the analogue, for the action of $G$ on $T_H$ given by the morphism $\eta$, of the morphism $\xi$
(clearly, such action of $G$ on $T_H$ is $H^1$-trivial).
Since $\phi^*$ is an isomorphism
and $\rho$ is characterized by the condition that $\rho(g)(h)=\xi(h)(g)$,
it suffices to prove that $\eta(g)(h)=\xi_H(\epsilon(h))(g)$ for every $h\in H$
(recall that $\epsilon$ is the inverse of $\phi^*$).
Since $G$ acts on $T_H$ by translations through the morphism $\eta$, we may choose a
group isomorphism $T_H\simeq (S^1)^b$ and identify the action of $G$ on $T_H$ with
the diagonal action for a collection of actions of $G$ on each of the factors $S^1$,
given by morphisms $\eta_1,\dots,\eta_b:G\to S^1$. This reduces the proof to checking that,
if $G$ acts on $S^1$ via translations through a morphism $\zeta:G\to S^1$, then
the composition
$$H^1(S^1)\simeq E_2^{0,1}\stackrel{d_2^{0,1}}{\longrightarrow}
E_2^{2,0}=H^2(BG;H^0(S^1))\simeq H^2(G)\stackrel{\mu}{\longrightarrow}\Hom(G,S^1)$$
sends the generator $\theta\in H^1(S^1)$ to $\zeta$. In the previous sequence,
$d_2^{0,1}$ is a differential in the spectral sequence for the Borel fibration $(S^1)_G\to BG$.
This fibration is a circle bundle over $BG$, so $d_2^{0,1}(\theta)=c_1(EG\times_{\zeta}S^1)$, by the relation
between Serre's spectral sequence and Gysin's sequence of a circle bundle.
To finish, it suffices to observe that the map
$\Hom(G,S^1)\to H^2(G)$, sending $\zeta:G\to S^1$ to $c_1(EG\times_{\zeta}S^1)$,
is the inverse of $\mu$.
\end{proof}

The following corollary will be used
later in the proof of Theorem \ref{thm:fixed-points-kernel-rotation-morphism}.

\begin{corollary}
\label{cor:trivial-rotation-pullback-H-1}
Let a finite group $G$ act $H^1$-trivially on a closed connected topological manifold $X$.
Assume the rotation morphism of this action is trivial.
Let $\zeta:X\to X':=X/G$ be the quotient map.
Then $\zeta^*:H^1(X')\to H^1(X)$ is surjective.
\end{corollary}
\begin{proof}
By the theorem, there is a map $\phi:X\to T_H$ such that
$\phi^*:H^1(T_H)\to H^1(X)$ is an isomorphism, and which is equivariant
with respect to the rotation morphism $\rho:G\to T_H$, i.e.,
$\phi(g\cdot x)=\rho(g)+\phi(x)$ for $x\in X$ and $g\in G$.
Since by assumption $\rho$ is trivial, this implies that $\phi:X\to T_H$ factors
through $\zeta:X\to X'$.
\end{proof}

\section{Proof of Theorem \ref{thm:main-rotation-morphism}}
\label{s:proof-thm:main-rotation-morphism}

Let us fix a closed connected $n$-dimensional topological manifold $X$
and denote as before $H:=H^1(X)$.
%
Denote by
$$\kK(X)$$
the set of all subgroups $K\leq H$ for which:
\begin{enumerate}
\item[(P1)] $H^*(X_K)$ is finitely generated as a $\ZZ$-module,
\item[(P2)] there is a subgroup of finite index $K^*_0\leq K^*$ whose induced action
on $H^*(X_K)$ is trivial,
\item[(P3)] the Serre spectral sequence over the rationals of the fibration
$X_K\times_{K^*_0} V_K\to V_K/K^*_0$ degenerates in the second page, and
\item[(P4)] $\psi_K^*:H^*(T_K)\to H^*(X)$ is injective.
\end{enumerate}
The set $\kK(X)$ is not empty, as it contains the trivial subgroup $0\leq H$.

Recall that for any subgroup $K\leq H$ we denote by
$r_K:T_H\to T_K$ the morphism induced by restriction.
Define, for each $H^1$-trivial action of a finite group $G$ on $X$,
$$C(G,X):=\min\{|\rho(G)\cap\Ker r_K| \mid K\in\kK(X)\},$$
where $\rho:G\to T_H$ denotes the rotation morphism of the action of $G$ on $X$.

Fix a number $C_1$. We want to prove that
there is a number $C_2$, depending only on $X$ and $C_1$, such that, for
every prime $p$ and every $H^1$-trivial action of a finite $p$-group $G$
on $X$ satisfying $|\Ker\rho|\leq C_1$, we have $C(G,X)\leq C_2$ .

\subsection{Introducing the sequence of groups}
\label{ss:introducing-sequence-groups}
Arguing by contradiction,
let us assume that there exists a sequence of (non necessarily pairwise distinct)
primes $p_1,p_2,\dots$, and groups $G_1,G_2,\dots$, where $G_i$ is a $p_i$-group,
and an $H^1$-trivial action of each $G_i$ on $X$ with rotation morphism $\rho_i:G_i\to T_H$,
satisfying $|\Ker\rho_i|\leq C_1$ for each $i$, and such that
\begin{equation}
\label{eq:C-G-i-X-diverges}
C(G_i,X)\to\infty.
\end{equation}
We will see that this leads to a contradiction.
Denote by
$$\phi_i:X\to T_H$$
a map representing the canonical element in $[X,T_H]$ and satisfying
\begin{equation}
\label{eq:equivariance}
\phi_i(g\cdot x)=\rho_i(g)+\phi_i(x)\qquad\qquad\text{for every $x\in X$ and $g\in G_i$,}
\end{equation}
as given by Theorem \ref{thm:rotation-morphism}.

We say that a subgroup $K\leq H$ is {\it cofree} if $H/K$ is torsion free.

\begin{lemma}
\label{lemma:descending-chain-stabilizes}
Any descending chain of cofree subgroups of $H$ stabilizes.
\end{lemma}
\begin{proof}
For any finitely generated abelian group $A$ we define
$b(A):=\dim_{\QQ}A\otimes_{\ZZ}\QQ.$
If $K\leq K'$ is an inclusion of cofree subgroups of $H$, then
$K'/K\leq H/K$, so $K'/K$ is torsion free.
Hence there is a section $s:K'/K\to K'$ of the projection map $K'\to K'/K$.
We have $K'=K\oplus s(K'/K)$, so
$b(K')=b(K)+b(s(K'/K))=b(K)+b(K'/K)$. This implies
that if $K\leq K'$ are cofree subgroups of $H$ then $b(K)\leq b(K')$,
with equality if and only if $K=K'$.
This implies the lemma.
\end{proof}

Define
\begin{equation}
\label{eq:def-kK}
\lL=\{K\leq H\mid K\text{ cofree, and }\liminf_{i\to\infty}|\rho_{i}(G_{i})\cap \Ker r_K|<\infty\}.
\end{equation}
Clearly $H\in\lL$, so $\lL$ is nonempty.
Choose a minimal element
$$K\in\lL$$ with respect to inclusion.
Such minimal element exists by Lemma \ref{lemma:descending-chain-stabilizes}.

Choose any sequence of naturals $i_1<i_2<i_3<\dots$ such that
$|\rho_{i_j}(G_{i_j}) \cap \Ker r_K|$ is bounded as $j\to\infty$.
We replace the sequence $(G_i)_i$ by the subsequence $(G_{i_j})_j\subseteq(G_i)$. In other words,
we redefine
$G_j:=G_{i_j}$
and
$\phi_j:=\phi_{i_j}$, $\rho_j:=\rho_{i_j}$ and $p_j:=p_{i_j}$.

In the remainder of this section we are going to prove that $K\in\kK(X)$.
This will give the desired contradiction, as we will see in Subsection
\ref{ss:conclusion-thm:main-rotation-morphism}.
Property (P1) will be proved in
Lemma \ref{lemma:finite-generated-cohomology-over-integers} using the main technical result in
\cite{Mu2021};
property (P2) will be proved in \S\ref{ss:def-K-0}, using the result in Section \ref{s:proof-thm:trivial-module-structure}.
The proof of property (P3) is the most involved result in this section. It has a geometric and an algebraic part.
Subsections \ref{ss:preparing-ss}, \ref{ss:groups-reenter} and \ref{ss:diagram-fibrations}
contain the geometric part of the argument. Subsection \ref{ss:sketch-argument}
sketches the algebraic part; some readers might prefer to begin reading this subsection after the proof of (P2).
Subsection \ref{ss:ignoring-torsion} develops the tools for the algebraic part, and
the argument is concluded in Subsection \ref{ss:conclusion-proof} (see Theorem \ref{thm:degeneracio-ssSerre}).
Finally, (P4) will follow from (P3) and the discussion in Subsection \ref{ss:preparing-ss}.

\subsection{Subgroups of $\rho_{K,i}(G_i)$ isomorphic to $(\ZZ/m_i)^k$, where $K\simeq\ZZ^k$}
\label{ss:subgroups-of-rho-K-i-G-i}

Let
$$C:=\sup_i |\rho_i(G_i)\cap \Ker r_{K}|.$$
Denote
$$\rho_{K,i}:=r_K\circ \rho_i:G_i\to T_K$$
We have
\begin{equation}
\label{eq:size-Gamma-d-i}
|\rho_{K,i}(G_i)|=\frac{|\rho_i(G_i)|}{|\rho_i(G_i)\cap\Ker r_K|}\geq |\rho_i(G_i)|/C,
\end{equation}
which, combined with the bound $|\Ker\rho_i|\leq C_1$, implies
\begin{equation}
\label{eq:una-estrella}
|\Ker\rho_{K,i}|\leq C_2:=C\cdot C_1.
\end{equation}

Let $k=\dim_{\QQ}K\otimes_{\ZZ}\QQ$,
so that $K\simeq\ZZ^k$. Choose an isomorphism
$$\eta:\ZZ^k\to K.$$
Denote by $e:T^k\to\Hom(\ZZ^k,S^1)$ the isomorphism
sending $(\theta_1,\dots,\theta_k)\in T^k=\RR^k/\ZZ^k$ to the morphism
$\ZZ^k\ni (a_1,\dots,a_k)\mapsto\sum_k a_i\theta_i\in\RR/\ZZ=S^1$.
The morphism of groups
$$\xi:T_{K}\stackrel{\simeq}{\longrightarrow} T^k$$
sending $\chi\in T_{K}$ to $e^{-1}(\chi\circ\eta)$ is an isomorphism.
Denote by $S_j$
the subgroup of $T^k$ given by the $j$-th factor, that is,
$$S_j=\{(\theta_1,\dots,\theta_k)\in T^r\mid \theta_i=0\text{ for every $i\neq j$}\}.$$

\begin{lemma}
\label{lemma:large-cyclic-groups}
For each $j$ we have
$|\rho_{K,i}(G_i)\cap \xi^{-1}(S_j)|\to\infty$ as $i\to\infty$.
\end{lemma}
\begin{proof}
We argue by contradiction.
Suppose that for some $j$ there exists $C'$ and a sequence
$i_v\to\infty$ such that
$|\rho_{K,{i_v}}(G_{i_v})\cap \xi^{-1}(S_j)|\leq C'$ for every $v$.
Define 
$$K'=\eta(\{(a_1,\dots,a_r)\in\ZZ^k\mid a_j=0\})\leq K.$$
We have $K/K'\simeq\ZZ$ and there is an exact sequence
$$0\to K/K'\to H/K'\to H/K\to 0.$$
Since both  $K/K'$ and $H/K$ are free abelian, so is $H/K'$, and consequently
$K'$ is cofree.
Denote by $r_{K,K'}:T_{K}\to T_{K'}$ the restriction map.
Using
(\ref{eq:size-Gamma-d-i}) we may bound:
\begin{multline*}
|\rho_{K',{i_v}}(G_{i_v})|=\frac{|\rho_{K,{i_v}}(G_{i_v})|}{|\rho_{K,{i_v}}(G_{i_v})\cap \Ker r_{K,K'}|}=
\frac{|\rho_{K,{i_v}}(G_{i_v})|}{|\rho_{K,{i_v}}(G_{i_v})\cap \xi^{-1}(S_j)|}\geq \\
\geq
\frac{|\rho_{K,{i_v}}(G_{i_v})|}{C'}\geq\frac{|\rho_{K,{i_v}}(G_{i_v})|}{CC'}.
\end{multline*}
This implies, by (\ref{eq:size-Gamma-d-i}), that
$|\rho_{K,{i_v}}(G_{i_v})\cap \Ker r_{K'}|\leq CC'$ for every $v$,
and consequently that $K'\in\lL$. This contradicts the minimality of $K$, so our initial
assumption cannot hold and the proof of the lemma is complete.
\end{proof}

For each integer $m$ we denote by $T_K[m]$ the $m$-torsion of $T_K$.

\begin{lemma}
\label{lemma:discrete-tori-in-rho-G}
There is a sequence of integers $m_1,m_2,\dots$, with $m_i\to\infty$,
such that $T_K[m_i]\leq \rho_{K,i}(G_i)$ for each $i$.
\end{lemma}
\begin{proof}
Let $m_i$ be the GCD of the numbers $|\rho_{K,i}(G_i)\cap\xi^{-1}(S_j)|$
for $j=1,\dots,k$. Then $T_K[m_i]\leq\rho_{K,i}(G_i)$. Since $G_i$
is a $p_i$-group, so is $\rho_{K,i}(G_i)$, which is a quotient of $G_i$.
So each number $|\rho_{K,i}(G_i)\cap\xi^{-1}(S_j)|$ is a power of $p_i$.
Consequently, $m_i=\min_j |\rho_{K,i}(G_i)\cap\xi^{-1}(S_j)|$. By
Lemma \ref{lemma:large-cyclic-groups}, it follows that $m_i\to\infty$.
\end{proof}

\subsection{An abelian cover of $X$ with finitely generated homology}
\label{ss:abelian-covers}
We begin recalling some definitions and results from \cite[\S 8]{Mu2021}.
Denote by $\pi:\RR^k\to T^k=\RR^k/\ZZ^k$ the quotient map.
Given any continuous map $\zeta:X\to T^k$ we define
$$X_{\zeta}=\{(x,u)\in X\times\RR^k\mid \zeta(x)=\pi(u)\}.$$
The projection
$$\pi_{\zeta}:X_{\zeta}\to X,\qquad \pi_{\zeta}(x,u)=x$$
is an unramified covering map and can be seen as
a principal $\ZZ^k$-bundle, where the action of $\ZZ^k$ on $X_{\zeta}$
is defined, for $\nu\in \ZZ^k$ and $(x,u)\in X_{\zeta}$,
as $\nu\cdot (x,u)=(x,u+\nu)$. Of course, $\pi_{\zeta}$ is the pullback
of the $\ZZ^k$-bundle $\pi$ via $\zeta$. Standard results on fiber bundles imply
the following (see \cite[Lemma 8.1]{Mu2021}).

\begin{lemma}
\label{lemma:homotopic-bundles-isomorphic}
If two continuous maps $\zeta,\zeta':X\to T^k$ are homotopic then there
is a $\ZZ^k$-equivariant homeomorphism $h:X_{\phi}\to X_{\phi'}$
such that $\pi_{\zeta}=\pi_{\zeta'}\circ h$.
\end{lemma}

The action of $\ZZ^k$ on $X_{\phi}$ induces an action of $\ZZ^k$ on $H_*(X_{\phi})$
by group automorphisms, which defines on $H_*(X_{\phi})$ a structure of
$\ZZ[\ZZ^k]$-module. 
This is \cite[Lemma 9.2]{Mu2021}:

\begin{lemma}
\label{lemma:finite-generated-cohomology-over-group-ring}
$H_*(X_{\phi})$ is finitely generated as a $\ZZ[\ZZ^k]$-module.
\end{lemma}

Choose a map $\psi:X\to T_H$ representing the canonical element in $[X,T_H]$ (see Lemma \ref{lemma:canonical-class}).
This map will be fixed for the remainder of this section.
Let $\psi_K:=r_K\circ \psi$.
Define
$$\Psi=\xi\circ \psi_K:X\to T^k.$$
(recall that the isomorphism
$\xi:T_K\to T^k$ is defined in Subsection \ref{ss:subgroups-of-rho-K-i-G-i}).
As explained above, $X_{\Psi}$ carries a free action of
$\ZZ^k$.

We identify the group ring $\ZZ[\ZZ^k]$ with the additive group $\Map_0(\ZZ^k,\ZZ)$
of finitely supported maps $\ZZ^k\to\ZZ$,
with the ring structure given by convolution.
Namely, the product of two elements $f,g\in\Map_0(\ZZ^k,\ZZ)$ is the map
$fg\in\Map_0(\ZZ^k,\ZZ)$ defined by the condition that, for every $v\in\ZZ^k$,
$$fg(v)=\sum_{u\in\ZZ^k}f(u)g(v-u).$$

Let $e_1,\dots,e_k$ denote the canonical basis of $\ZZ^k$,
and let $t_i:\ZZ^k\to\ZZ$ denote the characteristic function of $\{e_i\}\subset\ZZ^k$.
Then $t_i\in\Map_0(\ZZ^k,\ZZ)=\ZZ[\ZZ^k]$, and
the natural morphism
$$A:=\ZZ[t_1^{\pm 1},\dots,t_k^{\pm 1}]\to\ZZ[\ZZ^k]$$
is an isomorphism of rings. So the action of $\ZZ^k$ on $X_\Psi$
allows us to look at $H^*(X_{\Psi})$ as an $A$-module.

Let
\begin{equation}
\label{eq:isomorfisme-theta}
\theta:K^*\to\ZZ^k
\end{equation}
be the composition of
$\eta^*:K^*=\Hom(K,\ZZ)\to\Hom(\ZZ^k,\ZZ)$ with the isomorphism
$\Hom(\ZZ^k,\ZZ)\to\ZZ^k$ sending any $\alpha\in\Hom(\ZZ^k,\ZZ)$
to $(\alpha(e_1),\dots,\alpha(e_k))$. Then $\theta$ is an isomorphism of groups.

Since $\xi$ is an isomorphism of tori, and since $X_K=X_{\psi_K}$, we have a homeomorphism
\begin{equation}
\label{eq:X-Psi-X-K}
X_K\stackrel{\cong}{\longrightarrow} X_\Psi,
\end{equation}
which is equivariant with respect to the morphism $\theta$ defined above (\ref{eq:isomorfisme-theta}).

\begin{lemma}
\label{lemma:arrels-de-t-j}
For every $1\leq j\leq k$ there exists a nonzero integer $d_j$,
a sequence of integers $r_{j,i}$ satisfying $r_{j,i}\to\infty$ as $i\to\infty$,
and $A$-module automorphisms $w_{j,i}:H_*(X_{\Psi})\to H_*(X_{\Psi})$
such that $w_{j,i}^{r_{j,i}}$ coincides with multiplication by $t_j^{d_j}$.
\end{lemma}
\begin{proof}
Fix some integer $1\leq j\leq k$. By Lemma \ref{lemma:large-cyclic-groups}, we have
$$s_{j,i}:=|\rho_{K,i}(G_i)\cap \xi^{-1}(S_j)|\to\infty\quad\text{as}\quad i\to\infty.$$
The maps
$$\Phi_i:=\xi\circ r_{K}\circ\phi_i:X\to T^k$$
are homotopic to $\Psi$ for every $i$, because $\Psi=\xi\circ r_K\circ\psi$, where
$\psi$ represents the canonical element in $[X,T_H]$, and hence $\phi_i$
and $\psi$ are homotopic.
Hence, Lemma \ref{lemma:homotopic-bundles-isomorphic} implies the existence of isomorphisms
\begin{equation}
\label{eq:X-Psi-X-Phi-isomorfes}
H_*(X_{\Psi})\simeq H_*(X_{\Phi_i})\qquad\text{as $A$-modules}.
\end{equation}
The group $\rho_{K,i}(G_i)\cap \xi^{-1}(S_j)$ is cyclic, because it is a finite subgroup
of $\xi^{-1}(S_j)\simeq S^1$. Let $\gamma_{j,i}$ be the generator of $\rho_{K,i}(G_i)\cap \xi^{-1}(S_j)$
characterized by the equality
$$\xi(\gamma_{j,i})=[e_j/s_{j,i}],$$
where, for any $\alpha\in\RR^k$, $[\alpha]$ is the class of $\alpha$ in $\RR^k/\ZZ^k=T^k$.
Choose $g_{j,i}\in G_{i}$ such that
$$\gamma_{j,i}=\rho_{K,i}(g_{j,i}).$$
Let
$o_{j,i}\in\NN$
be the order of $g_{j,i}$, which is a multiple of $s_{j,i}$.
Since $g_{j,i}^{s_{j,i}}\in\Ker \rho_{K,i}$, by (\ref{eq:una-estrella})
the order of $g_{j,i}^{s_{j,i}}$ is not bigger than $C_2$.
Consequently,
\begin{equation}
\label{eq:dues-estrelles}
o_{j,i}/s_{j,i}\leq C_2.
\end{equation}
Let
$f_{j,i}:X\to X$ be the map defined as
$f_{j,i}(x)=g_{j,i}\cdot x$ for every $x$. 
Define a lift of $f_{j,i}$,
$$F_{j,i}:X_{\Psi}\to X_{\Psi},$$
by setting
$$F_{j,i}(x,u)=\left(f_{j,i}(x),u+\frac{e_j}{s_{j,i}}\right)$$
(recall that $X_{\Psi}$ is a subset of $X\times\RR^k$).
The previous expression
does indeed define an automorphism of $X_{\Psi}$, because if $(x,u)\in X\times\RR^k$
belongs to $X_\Psi$ then, by (\ref{eq:equivariance}), so does $F_{j,i}(x,u)$.
The map $F_{j,i}$ commutes with the action of $\ZZ^k$ on $X_{\Psi}$. Hence, if we denote by
$$v_{j,i}:H_*(X_{\Phi_i})\to H_*(X_{\Phi_i})$$
the morphism induced by $F_{j,i}$, then $v_{j,i}$ is a morphism of
$A$-modules. It follows from the definition of
$F_{j,i}$ that
$v_{j,i}^{o_{j,i}}$ coincides with the map given by multiplication by $t_j^{o_{j,i}/s_{j,i}}$.
We have
$\Phi_i\circ f_{j,i}^{s_{j,i}}=\Phi_i$. This implies that
the map
$$h_{j,i}:X\times\RR^k\to X\times\RR^k,\qquad h_{j,i}(x,u)=\left(f_{j,i}^{s_{j,i}}(x),u\right)$$
preserves $X_{\Phi_i}$ and commutes with the $\ZZ^k$ action, and hence induces a morphisms
of $A$-modules
$$z_{j,i}:H_*(X_{\Phi_i})\to H_*(X_{\Phi_i})$$
which can be identified with $t_j^{-1}v_{j,i}^{s_{j,i}}$. Indeed, if we denote by
$\tau_j:X\times \RR^k\to X\times \RR^k$ the translation $\tau_j(x,u)=(x,u+e_j)$, then
$\tau_j$ preserves $X_{\Phi_i}$ and we have
$$F_{j,i}^{s_{j,i}}=\tau_j\circ h_{j,i}.$$
The morphism of $H_*(X_{\Phi_i})$ induced by $\tau_j$ is multiplication by $t_j$,
so $v_{j,i}^{s_{j,i}}=t_j z_{j,i}$, or equivalently $z_{j,i}=t_j^{-1}v_{j,i}^{s_{j,i}}$.
Finally, (\ref{eq:dues-estrelles}) implies that
$$z_{j,i}^{C_2!}=\left(t_j^{-1}v_{j,i}^{s_{j,i}}\right)^{C_2!}=1$$
for every $i$. This now allows us to write, setting $d_j:=C_2!$ and
$r_{j,i}=s_{j,i}C_2!$,
$$v_{j,i}^{r_{j,i}}=t_j^{d_j}.$$
Via the isomorphism (\ref{eq:X-Psi-X-Phi-isomorfes}), $v_{j,i}$ corresponds
to an automorphism of $A$-modules
$$w_{j,i}:H_*(X_{\Psi})\to H_*(X_{\Psi})$$
satisfying $w_{j,i}^{r_{j,i}}=t_j^{d_j}$,
so the proof of the lemma is complete.
\end{proof}


The following lemma implies that $K$ satisfies property (P1) in the definition of $\kK(X)$:

\begin{lemma}
\label{lemma:finite-generated-cohomology-over-integers}
$H^*(X_K)$ is finitely generated as a $\ZZ$-module.
\end{lemma}
\begin{proof}
Combining the previous lemma with \cite[Corollary 6.3]{Mu2021}, we deduce that
$H_*(X_{\Psi})\simeq H_*(X_K)$ is finitely generated as a $\ZZ$-module.
By the universal coefficients theorem, it follows that $H^*(X_K)$ is also
a finitely generated $\ZZ$-module.
\end{proof}

Arguing as in the proof of \cite[Lemma 2.6]{Mu2016}, we conclude:

\begin{lemma}
\label{lemma:Minkowski-X-K}
There is a constant $C_{\Min}$ such that any finite subgroup $G\leq\Aut(H^*(X_K))$
satisfies $|G|\leq C_{\Min}$.
\end{lemma}
\begin{proof}
We prove, more generally, that if $M$ is a finitely generated $\ZZ$-module
then there exists a uniform upper bound on the size of the finite subgroups of $\Aut M$. Since $M$ is finitely generated, its torsion submodule $\Tor M$ is finite,
and we have an isomorphism $\phi:M/\Tor M\simeq\ZZ^r$ for some natural number $r$. Since any automorphism of $M$ preserves $\Tor M$, there is a natural morphism of groups $\Aut M\to\Aut (M/\Tor M)$ which, composed with conjugation by $\phi$, gives a morphism $\rho:\Aut M\to\GL(r,\ZZ)$.
By Minkowski's lemma \cite{Min,Serre},
there exists a uniform upper bound on the size of the finite subgroups of $\GL(r,\ZZ)$.
Consequently, it suffices to prove the existence of a uniform upper bound on the size of the finite subgroups of $\Ker\rho$. Any $\psi\in\Ker\rho$ is an automorphism $\psi:M\to M$ satisfying $\psi(p)-p\in\Tor M$ for every $p\in M$. Hence we may define an isomorphism of groups $\Ker\rho\to\Hom(M,\Tor M)$ by sending $\psi$ to the morphism $M\ni p\mapsto \psi(p)-p\in\Tor M$. But $\Hom(M,\Tor M)$
is finite (because $M$ is finitely generated), so $\Ker\rho$ is also finite.
\end{proof}

\subsection{Defining the subgroup $K_0^*$}
\label{ss:def-K-0}

We are going to use the following result, whose
proof is postponed to Section \ref{s:proof-thm:trivial-module-structure}
to avoid breaking the flow of the argument.

\begin{theorem}
\label{thm:trivial-module-structure}
Let $M$ be a finitely generated $\ZZ$-module and let $z:M\to M$ be an automorphism.
Suppose that there is a sequence of integers $r_i\to\infty$ and automorphisms $w_i:M\to M$
satisfying $w_i^{r_i}=z$. Then $z$ has finite order in $\Aut M$.
\end{theorem}

Combining Lemmas \ref{lemma:arrels-de-t-j} and \ref{lemma:finite-generated-cohomology-over-integers}
with Theorem \ref{thm:trivial-module-structure}, we deduce the existence of a number $e$ such that,
for each $j$, multiplication by $t_j^e$ is the identity on $H^*(X_{\Psi})$.
Define
$$K_0^*=\{f:K\to\ZZ\mid f(K)\subseteq e\ZZ\}\subseteq K^*=\Hom(K,\ZZ).$$
Then $K_0^*$ has finite index in $K^*$. Since the homeomorphism
$X_K\stackrel{\cong}{\longrightarrow} X_\Psi$
defined in (\ref{eq:X-Psi-X-K}) is equivariant with respect to the
isomorphism $\theta:K^*\to\ZZ^k$ given in (\ref{eq:isomorfisme-theta}),
it follows that $K_0^*$ acts trivially on $H^*(X_K)$.

We have thus proved that $K$ satisfies property (P2) in the definition of $\kK(X)$.

\subsection{Preparing a spectral sequence argument}
\label{ss:preparing-ss}
In this subsection we are going to prove that, if $K$ and $K_0^*\leq K^*$ satisfy property
(P3) in the definition of $\kK(X)$, then they also satisfy (P4).

Let $K^*$ act on $X_K\times V_K$ by
$\lambda\cdot(x,v)=(\lambda\cdot x,v-\lambda)$ for each $\lambda\in K^*$,
and let $X_K\times_{K^*}V_K$ be the quotient space. There are natural projection maps
$$T_K=V_K/K^*\stackrel{\Pi}{\longleftarrow} X_K\times_{K^*}V_K\stackrel{\Theta}{\longrightarrow} X_K/K^*=X.$$
Both $\Pi$ and $\Theta$ are locally trivial fibrations, with fibers $X_K$ and $V_K$ respectively.
In particular, $\Theta$ is a homotopy equivalence.

\begin{lemma}
\label{lemma:homotopy-commutative-diagram}
The following diagram is homotopy commutative:
$$\xymatrix{X_K\times_{K^*}V_K\ar[rr]^-{\Theta}\ar[rd]_{\Pi} & & X\ar[ld]^{-\psi_K} \\
& T_K}$$
\end{lemma}
\begin{proof}
Let $V_K\times_{K^*}V_K$ be the quotient of $V_K\times V_K$ under the action of $K^*$ given
by $\lambda\cdot(u,v)=(u+\lambda,v-\lambda)$. Denote by $[u,v]$ the class of $(u,v)\in V_K\times V_K$
in $V_K\times_{K^*}V_K$.
Let $\pi_L,\pi_R:V_K\times_{K^*}V_K\to T_K$ be the projections
$\pi_L([u,v])=[u]$ and $\pi_R([u,v])=[v]$.
Let $\Delta_K=\{[u,v]\in V_K\times_{K^*}V_K\mid u+v=0\}$.
The inclusion $\Delta_K\hookrightarrow V_K\times_{K^*}V_K$ is a homotopy equivalence,
with homotopy inverse $[u,v]\mapsto[(u-v)/2,(v-u)/2]$.
Since $\pi_L|_{\Delta_K}=-\pi_R|_{\Delta_K}$, it follows that
$\pi_L\sim -\pi_R.$

Denote by $\Psi_K:X_K\to V_K$ the natural projection,
and denote by $\Pi_L:X_K\times_{K^*}V_K\to T_K$ the map $\Pi_L([x,v])=[\Psi_K(x)]$.
The map $\Psi_K$ is $K^*$-equivariant, and
it induces a map $\xi_K:X_K\times_{K^*}V_K\to V_K\times_{K^*}V_K$ yielding
commutative diagrams
$$\xymatrix{X_K\times_{K^*}V_K\ar[r]^{\xi_K}\ar[d]_{\Pi} & V_K\times_{K^*}V_K \ar[d]^{\pi_R} \\
T_K\ar@{=}[r] & T_K,}\qquad
\xymatrix{X_K\times_{K^*}V_K\ar[r]^{\xi_K}\ar[d]_{\Pi_L} & V_K\times_{K^*}V_K \ar[d]^{\pi_L} \\
T_K\ar@{=}[r] & T_K.}$$
The diagrams, combined with $\pi_L\sim -\pi_R$, imply that $\Pi_L\sim -\Pi$. Since
$\Pi_L=\psi_K\circ\Theta$, it follows that
$\Pi\sim (-\psi_K)\circ\Theta,$
which is what we wanted to prove.
\end{proof}

Consequently, the injectivity of the map
$\psi_K^*:H^*(T_K)\to H^*(X)$
is equivalent to the injectivity of
\begin{equation}
\label{eq:Pi-cohomology}
\Pi^*:H^*(T_K)\to H^*(X_K\times_{K^*}V_K).
\end{equation}
Furthermore, $\Pi\sim (-\psi_K)\circ\Theta$ implies
that $X_K$ is the homotopy fiber of the
map $\psi_K:X\to T_K$.

Let $\pi:E\to B$ be a locally trivial fibration. We say that $\pi$ {\it has trivial monodromy}
if, for every $p\in B$, the monodromy action on the cohomology of the fibers, $\pi_1(B,p)\to\Aut H^*(\pi^{-1}(p))$, is trivial.

The monodromy
action on the cohomology of the fiber of $\Pi$ at any point $p\in T_K$
can be identified with the morphism $K^*\to\Aut H^*(X_K)$ via the
canonical isomorphism $\pi_1(T_K,p)\simeq K^*$.
Hence, the fibration
$$\Pi_0:X_K\times_{K^*_0}V_K\to V_K/K_0^*=:T_{K_0}$$
has trivial monodromy.
We have a Cartesian diagram
$$\xymatrix{X_K\times_{K_0^*}V_K \ar[r]\ar[d]_{\Pi_0} & X_K\times_{K^*}V_K \ar[d]^{\Pi} \\
T_{K_0} \ar[r]^{\pi_0} & T_K,}$$
where both horizontal maps are the quotient maps for the residual action of $K^*/K_0^*$.

The residual action of $K^*/K_0^*$ on $T_{K_0}$ is free, because it is an action through translations.
Hence, $\pi_0^*:H^1(T_{K})\to H^1(T_{K_0})$ is injective. Since
both $T_K$ and $T_{K_0}$ are tori, their cohomology can be identified with the exterior algebra
of their $H^1$, so $\pi_0^*:H^*(T_K)\to H^*(T_{K_0})$ is injective as well.
Consequently, if we prove that
\begin{equation}
\label{eq:Pi_0-cohomology}
\Pi_0^*:H^*(T_{K_0})\to H^*(X_K\times_{K_0^*}V_K)
\end{equation}
is injective,  then it will follow that the morphism (\ref{eq:Pi-cohomology}) is also injective.

Since $H^*(T_{K_0})$ is torsion free, $\Pi_0^*$ is injective if and only if the analogous map
$\Pi_0^*:H^*(T_{K_0};\QQ)\to H^*(X_K\times_{K_0^*}V_K;\QQ)$ is injective. Hence,
the injectivity of (\ref{eq:Pi_0-cohomology}) will follow from the fact that the Serre spectral
sequence for $\Pi_0$ {\it over the rationals} degenerates at the second page. The remainder of
this section is devoted to proving this degeneracy. We do it
by comparing the fibration $\Pi_0$ with a collection of fibrations defined using the actions of
the groups $G_i$ on $X$ (see the diagram (\ref{eq:diagrama-fibracions}) below).

\subsection{Group actions reenter the stage}
\label{ss:groups-reenter}
Define
$$\phi_{K,i}:=r_K\circ\phi_i:X\to T_K$$
and
$$X_{K,i}:=X_{\phi_{K,i}}.$$
Since $\phi_i\sim\psi$, we have $\phi_{K,i}\sim \psi_K$ for every $i$,
and consequently, by Lemma \ref{lemma:homotopic-bundles-isomorphic},
there is a $K^*$-equivariant homeomorphism
$$\zeta_i:X_{K,i}\to X_K.$$
In particular, $K_0^*$ acts trivially on $H^*(X_{K,i})$.
Let
$$J_i=\{(g,u)\in G_i\times V_K\mid \rho_{K,i}(g)=\pi_K(u)\}.$$
We have a commutative diagram
$$\xymatrix{G_i\ar[d]_{\rho_{K,i}} & J_i \ar[l]_{h_i} \ar[r]^{w_i} & V_K\ar[d]^{\pi_{K_0}} \\
T_K && T_{K_0} \ar[ll]_{\pi_{K,K_0}},}$$
where $h_i,w_i,\pi_{K_0},\pi_{K,K_0}$ are the obvious projection maps.
The following lemma, which will be later used, follows immediately from the definition of $J_i$.

\begin{lemma}
\label{lemma:Ker-w-i-Ker-rho}
We have $\Ker w_i=(\Ker \rho_{K,i})\times\{0\}\leq G_i\times V_K$.
\end{lemma}

The kernel of $h_i$
can be identified with the image of the
monomorphism $f_i:K^*\to J_i$ given by $f_i(u)=(1,u)$, so, abusing notation, we denote
it by $K^*$. Similarly, we denote $f_i(K_0^*)$ by the symbol $K_0^*$.
The projection $\pi_K:V_K\to T_K$ is surjective, so $h_i:J_i\to G_i$ is also surjective.
Hence, the quotient $J_i/K^*$ is isomorphic to $G_i$ and is consequently finite.
Since $[K^*:K_0^*]$ is finite, it follows that $J_i/K_0^*$ is also finite.

The equivariance property (\ref{eq:equivariance}) allows us to define an action
of $J_i$ on $X_{K,i}$ lifting the action of $G_i$ on $X$: if $(x,v)\in X_{K,i}$
and $(g,u)\in J_i$, then we set
\begin{equation}
\label{eq:accio-J-i-X-K-i}
(g,u)\cdot (x,v):=(g\cdot x,u+v).
\end{equation}
Since $K_0^*$ acts trivially on $H^*(X_{K,i})$, the action of $J_i$ on $H^*(X_{K,i})$
factors through the inclusion of a finite group in $\Aut H^*(X_{K,i})$, so by
Lemma \ref{lemma:Minkowski-X-K} the subgroup
$$L_i:=\Ker(J_i\to\Aut H^*(X_{K,i}))$$
satisfies
\begin{equation}
\label{eq:index-J-L}
[J_i:L_i]\leq C_{\Min}.
\end{equation}
By Lemma \ref{lemma:discrete-tori-in-rho-G} we may choose a sequence of naturals
$m_1,m_2,\dots$ satisfying $m_i\to\infty$, in such a way that
$T_K[m_i]\leq \rho_{K,i}(G_i)$. By the definition of $J_i$,
we have
$\rho_{K,i}\circ h_i = \pi_K\circ w_i.$
Since $h_i:J_i\to G_i$ is surjective, it follows that
$$T_K[m_i]\leq \pi_K(w_i(J_i)).$$
Identifying $K^*$ with a
subgroup of $J_i$ as explained above, we have
$$\pi_{K_0}(w_i(K^*))=K^*/K_0^*=\Ker\pi_{K,K_0},$$ and consequently
$$T_{K_0}[em_i]=\pi_{K,K_0}^{-1}(T_K[m_i])\leq \pi_{K_0}(w_i(J_i))$$
(the number $e$ was defined in Subsection \ref{ss:def-K-0}).
The bound (\ref{eq:index-J-L}) implies that
$$[T_{K_0}[em_i]:T_{K_0}[em_i]\cap \pi_{K_0}(w_i(L_i))]\leq C_{\Min}.$$
This inequality will be combined with the following lemma.

\begin{lemma}
Let $T=\RR^k/\ZZ^k$, and let $\Lambda\leq T[m]$ for some natural number $m$.
Suppose that $[T[m]:\Lambda]\leq C$. There is a natural number $m'$
satisfying $m'\geq m/C!$ and $T[m']\leq\Lambda$.
\end{lemma}
\begin{proof}
Let $q$ be the lcm of all divisors $d$ of $m$ such that $d\leq C$.
Clearly, $q\leq C!$. Let
$m'=m/q$, which is a natural number.
Let $g\in T[m]$ be any element, and let $[g]$ denote the class of $g$ in $T[m]/\Lambda$.
Then $\ord([g])$ divides $\ord(g)$, and $\ord(g)$ divides $m$, so $\ord([g])$ divides
$m$. Furthermore, $\ord([g])\leq |T[m]/\Lambda|\leq C$. Hence, $\ord([g])$ divides $q$.
It follows that $qT[m]\leq\Lambda$. But $qT[m]=T[m']$, so the lemma is proved.
\end{proof}

We may thus pick, for every $i$, a natural $n_i$ satisfying $n_i\geq em_i/C_{\Min}!$
(so $n_i\to\infty$), and such that
\begin{equation}
\label{eq:torsio-n-i}
T_{K_0}[n_i]\leq \pi_{K_0}(w_i(L_i)).
\end{equation}

\subsection{The diagram of fibrations}
\label{ss:diagram-fibrations}

Define
$$\Delta_i:=L_i\cap (\pi_{K_0}\circ w_i)^{-1}(T_{K_0}[n_i]).$$
Note that we have
$K_0^*\leq\Delta_i$. By (\ref{eq:torsio-n-i}), we have
\begin{equation}\label{eq:torsio-n-i-2}
\pi_{K_0}(w_i(\Delta_i))=T_{K_0}[n_i].
\end{equation}
Define an action of $\Delta_i$ on $X_{K,i}\times V_K$ by
$\delta\cdot(x,v)=(\delta\cdot x,v-w_i(\delta))$ for each $\delta\in\Delta_i$,
where in the first factor we use the restriction of the action of $J_i$ on $X_{K,i}$.
The quotient $X_{K,i}\times_{K_0^*}V_K:=(X_{K,i}\times V_K)/K_0^*$ carries a residual
action of $Q_i:=\Delta_i/K_0^*$.
Let
$$\Theta_i:=\Delta_i\cap\Ker w_i.$$
We have $\Theta_i\cap K_0^*=\{0\}$.
Hence, combining the inclusion $\Theta_i\hookrightarrow\Delta_i$ with the projection $\Delta_i\to\Delta_i/K_0^*$
we may identify $\Theta_i$ with a subgroup of $Q_i=\Delta_i/K_0^*$.
Consequently, $\Theta_i$ acts on $X_{K,i}\times_{K_0^*}V_K$ via an action on
$X_{K,i}\times V_K$ which is trivial on the second factor.


For any action of a group $G$ on a space $Z$, we denote as usual the corresponding
Borel construction by $Z_G=EG\times_GZ$, where $EG$ is the total space of the universal principal $G$-bundle $EG\to BG$ (hence, $EG$ is a contractible space on which $G$ acts freely),
and $EG\times_GZ$ is $(EG\times Z)/G$, the orbit space of $EG\times Z$ under the action of $G$ given by
$g\cdot(p,z)=(p\cdot g,g^{-1}\cdot z)$.
Since $EG\to BG$ is only unique up to homotopy equivalence, similarly the Borel
construction $Z_G$ is only uniquely defined up to homotopy equivalence
(or, equivalently, it is only well defined as a homotopy type).
If $G$ is a subgroup of a bigger group $H$ then, restricting to $G$ the action of $H$ on the total space of the universal fibration $EH\to BH$, we obtain a universal principal $G$-bundle
$EH\to EH/G$, so $Z_G$ is homotopy equivalent to $EH\times_GZ$. Finally, if $G$ is a normal subgroup of $H$ and the action of $G$ on $Z$ extends to an action of $H$, the space $EH\times_GZ$ inherits a residual action of $H/G$, whose orbit space is $Z_H=EH\times_HZ$.

Using the previous considerations, we can define an action of $Q_i/\Theta_i$ on a space representing the homotopy type $(X_{K,i})_{\Theta_i}\times_{K_0^*}V_K$. Indeed, the action of $\Theta_i$ on $X_{K,i}$ commutes with the action of $K_0^*$, so
\begin{multline*}
(X_{K,i})_{\Theta_i}\times_{K_0^*} V_K\simeq
((E\Delta_i \times X_{K,i})/\Theta_i\times V_K)/K_0^*
=((E\Delta_i \times X_{K,i}\times V_K)/\Theta_i)/K_0^*
=\\=((E\Delta_i \times X_{K,i}\times V_K)/K_0^*)/\Theta_i
=(E\Delta_i \times (X_{K,i}\times V_K)/K_0^*)/\Theta_i
\simeq(X_{K,i}\times_{K_0^*}V_K)_{\Theta_i}
\end{multline*}
(to understand the third and fourth terms, take into account that $\Theta_i$ acts trivially on $V_K$
and $K_0^*$ acts trivially on $E\Delta_i$).
Consequently, the residual action of $Q_i/\Theta_i\simeq (\Delta_i/K_0^*)/\Theta_i$ on $(X_{K,i}\times_{K_0^*}V_K)_{\Theta_i}$
can be transported to an action on a topological space which represents the homotopy type $(X_{K,i})_{\Theta_i}\times_{K_0^*} V_K$. When we write $(X_{K,i})_{\Theta_i}\times_{K_0^*} V_K$
in the sequel, we will mean {\it this} particular topological space, namely
$((E\Delta_i \times X_{K,i})/\Theta_i\times V_K)/K_0^*$.

The action of $\Delta_i$ on $V_K$ given by $\delta\cdot v=v-w_i(\delta)$ for each $\delta\in\Delta_i$
defines an action of $Q_i$ on $T_{K_0}=V_K/K_0^*$, with respect to which the projection map
$X_{K,i}\times_{K_0^*}V_K\to T_{K_0}$ is $Q_i$-invariant. The restriction to $\Theta_i$ of the action of $Q_i$
on $T_{K_0}$ is trivial, so the action of $Q_i$ defines an action of $Q_i/\Theta_i$
on $T_{K_0}$ with respect to which the projection map
$(X_{K,i})_{\Theta_i}\times_{K_0^*} V_K\to T_{K_0}$ is $Q_i/\Theta_i$-equivariant. Finally, by
(\ref{eq:torsio-n-i-2}), the action of $Q_i/\Theta_i$
on $T_{K_0}$ gives an isomorphism $Q_i/\Theta_i\simeq T_{K_0}[n_i]$.

The previous discussion justifies the commutativity of the leftmost square in the following commutative diagram:
\begin{equation}
\label{eq:diagrama-fibracions}
\xymatrix{
((X_{K,i})_{\Theta_i}\times_{K_0^*}V_K)/(Q_i/\Theta_i) \ar[d]_{\Pi_{3,i}} &
(X_{K,i})_{\Theta_i}\times_{K_0^*}V_K \ar[d]_{\Pi_{2,i}} \ar[l]_-{U_i} &
X_{K,i}\times_{K_0^*}V_K \ar[d]_{\Pi_{1,i}}\ar[r]^{Z_i}\ar@{_{(}->}[l]_-{W_i} &
X_K\times_{K_0^*}V_K \ar[d]_{\Pi_0} \\
T_{K_0}/T_{K_0}[n_i] &
T_{K_0}\ar@{=}[r]\ar[l]_{q_i} &
T_{K_0}\ar@{=}[r] &
T_{K_0},}
\end{equation}
where $V_i$ is the quotient map by
the free action of $Q_i/\Theta_i$, $W_i$ is induced by the inclusion
of $X_{K,i}$ inside the Borel construction $(X_{K,i})_{\Theta_i}$ as a fiber
of the projection
$$(X_{K,i})_{\Theta_i}=E\Delta_i \times_{\Theta_i}X_{K,i}\to E\Delta_i/\Theta_i,$$
$Z_i$ is induced by the equivariant homeomorphism $\zeta_i$,
map $q_i$ is the quotient map for the action of $Q_i/\Theta_i\simeq w_i(\Delta_i)/K_0^*\simeq T_{K_0}[n_i]$
on $T_{K_0}$ given by translations, and the vertical arrows are obtained by projecting
to $V_K$ and quotienting either by $w_i(\Delta_i)$ (first arrow) or by $K_0^*$ (remaining
arrows). The commutativity of the other two squares in the diagram is immediate.

All vertical arrows in the previous diagram are locally trivial fibrations with trivial monodromy.

\subsection{A sketch of the algebraic argument}
\label{ss:sketch-argument}
The previous diagram will be the key ingredient in the proof of the degeneracy of Serre's spectral
sequence. The arguments involved in this proof are
elementary but rather involved, so, for the reader's benefit,
we explain the main idea before delving into the details. Suppose that $Y\to T^k$ is a locally trivial fibration,
with fiber a space $F$ with finitely generated cohomology,
and with trivial monodromy. Suppose furthermore that there is a sequence of natural numbers
$n_i\to\infty$ and, for each $i$, a locally trivial fibration $Y_i\to T^k$ with fiber $F$ and with trivial monodromy,
and a cartesian diagram
$$\xymatrix{Y\ar[r]\ar[d] & Y_i \ar[d] \\ T^k\ar[r]^{q_i} & T^k,}$$
where $q_i$ is a covering satisfying $q_i^*H^1(T^k)=n_i H^1(T^k)$. It then follows that
$q_i^*H^j(T^k)=n_i^j H^j(T^k)$ for each $j$. This implies, by the naturality of Serre's spectral sequence, that
the differentials $d_2^{p,q}:E_2^{p,q}\to E_2^{p+2,q-1}$ in the second page are divisible by $n_i^2$. Since this is true for each $i$ and $n_i\to\infty$,
it follows that the image of $d_2^{p,q}$ is contained in the torsion of $E_2^{p,q}$,
because the assumption that $H^*(F)$ is finitely generated implies that $H^*(F)/\Tor H^*(F)$ is free. It follows that,
{\it up to torsion}, $E_3^{p,q}$ can be identified with a subgroup of $E_2^{p,q}$. We can repeat
the argument to prove that the image of $d_3^{p,q}$ is torsion, and so on. As a consequence,
and always up to torsion,
there are inclusions $E_2^{p,q}\supset E_3^{p,q}\supset E_4^{p,q}\supset\dots$,
and all differentials $d_r^{p,q}$ are torsion.
This implies the degeneracy of Serre's spectral sequence for $Y\to T^k$ over the rationals.

The previous idea can be applied to the leftmost square in (\ref{eq:diagrama-fibracions}), except that the fibration
$(X_{K,i})_{\Theta_i}\times_{K_0}^*V_K\to T_{K_0}$, which should play the role of $Y\to T^k$,
depends on $i$. But the groups $\Theta_i$ are uniformly bounded,
so that the entries in the Serre's spectral sequence of $(X_{K,i})_{\Theta_i}\times_{K_0}^*V_K\to T_{K_0}$
are equal to those of $X_{K_i}\times_{K_0}^*V_K\to T_{K_0}$ {\it up to uniformly bounded torsion}, and
the entries in the Serre's spectral sequence of $X_{K,i}\times_{K_0}^*V_K\to T_{K_0}$
are independent of $i$, thanks to the $K_0^*$-equivariant homeomorphisms
$\zeta_i:X_{K,i}\to X_K$ in Subsection \ref{ss:groups-reenter}. Note that
$H^*(X_K)$ is finitely generated, as we proved in Theorem \ref{lemma:finite-generated-cohomology-over-integers}.

In order to implement this strategy one has to keep track of each {\it up to torsion} in the previous sketch,
and that is the motivation for the contents of the following section. The proof of degeneracy is given afterwards,
in Subsection \ref{ss:conclusion-proof}.

\subsection{Ignoring the torsion in abelian groups}
\label{ss:ignoring-torsion}

\newcommand{\fexp}{\operatorname{fexp}}

We identify abelian groups with $\ZZ$-modules.
For any abelian group $M$ we denote $M_{\QQ}:=M\otimes_{\ZZ}\QQ$,
and $M_{\ZZ}$ denotes the image of the map $M\to M_{\QQ}$, $m\mapsto m\otimes 1$.
We have $M_{\ZZ}\simeq M/\Tor M$, so if $M$ is finitely generated then $M_{\ZZ}$ is free.
Any morphism of abelian groups $f:M\to N$ induces naturally morphisms
$f_{\QQ}:M_{\QQ}\to N_{\QQ}$ and $f_{\ZZ}:M_{\ZZ}\to N_{\ZZ}$, where
$f_{\ZZ}=f_{\QQ}|_{M_{\ZZ}}$.

Given a locally trivial fibration
$$\pi:E\to B$$
and a commutative ring with unit $R$ we denote by $E^{p,q}_k(\pi;R)$ the entries of
Serre's spectral sequence with coefficients in $R$ for the fibration $\pi$,
and we denote by
$$d^{p,q}_k(\pi;R):E^{p,q}_k(\pi;R)\to E^{p+k,q-k+1}_k(\pi;R)$$
the differentials.
Define $$E^{p,q}_k(\pi):=E^{p,q}_k(\pi;\ZZ).$$

\begin{lemma}
\label{lemma:tensoring-Q-spectral-sequence}
For any locally trivial fibration $\pi:E\to B$
and any $p,q,k$ there is an isomorphism
$$E^{p,q}_k(\pi;\QQ)\simeq E^{p,q}_k(\pi)_{\QQ}$$
which is natural with respect to morphisms of locally trivial fibrations.
Hence, for every $k\geq 2$ and every $p,q$,
$d^{p,q}_k(\pi;\QQ)=0$ if and only if $d^{p,q}_k(\pi)_{\QQ}=0$.
\end{lemma}

In particular, for any topological space $X$ there is a natural isomorphism
of rational vector spaces $H^*(X;\QQ)\simeq H^*(X)_{\QQ}$. This of course is
an immediate consequence of the universal coefficients theorem, which is actually
the only nontrivial ingredient in the proof of the lemma.

\begin{proof}
It suffices to prove the analogous statement for filtered complexes of abelian groups.
Let $C^*$ be a cohomological complex of abelian groups endowed with a decreasing filtration $F^*C^*$,
and let $E^{p,q}_k(F^*C^*)$ denote the terms in the spectral sequence associated to $F^*C^*$.
We want to prove that $E^{p,q}_k(F^*C^*_{\QQ})$ and $E^{p,q}_k(F^*C^*)_{\QQ}$ are naturally
isomorphic for every $p,q,k$. The definition of $E^{p,q}_k(F^*C^*)$
(see e.g. \cite[\S 2.2]{McCleary}) only involves the following operations: taking kernels and images of maps
of abelian groups, taking intersections of two subgroups of a given group, and taking the quotient of a group by a subgroup. All these operations commute with tensoring by $\QQ$, so the lemma follows.
\end{proof}

We say that a morphism $f:M\to N$ of abelian groups is a {\it $\QQ$-isomorphism} if
$f_{\QQ}$ is an isomorphism. If $f$ is a $\QQ$-isomorphism then $f_{\ZZ}$ is injective,
but not necessarily surjective.

Denote by
$$\hH(\pi;R)$$
the local system over $B$ given by the cohomology with $R$-coefficients
of the fibers of the locally trivial fibration $\pi:E\to B$.
Again, we write $\hH(\pi)$ instead of $\hH(\pi;\ZZ)$.

The following is a particular case of a general result of Hilton and Roitberg \cite{HiltonRoitberg}.
We include the proof for completeness.

\begin{lemma}
\label{lemma:isom-Serre-SS-rational}
Consider a commutative diagram of topological spaces
$$\xymatrix{E\ar[d]_{\pi}\ar[r]^{\phi} & E'\ar[d]^{\pi'} \\
B\ar[r]^{\psi} & B',}$$
where both $\pi$ and $\pi'$ are locally trivial fibrations and $B$ and $B'$ are connected.
Fix some $x\in B$ and denote $F=\pi^{-1}(x)$ and $F'=(\pi')^{-1}(\psi(x))$.
Suppose that:
\begin{enumerate}
\item the morphism $(\phi|_F)^*:H^*(F') \to H^*(F)$ is a $\QQ$-isomorphism,
\item both local systems $\hH(\pi)$ and $\hH(\pi')$ are trivial,
\item the morphism $\psi^*:H^*(B')\to H^*(B)$ is a $\QQ$-isomorphism.
\end{enumerate}
Then for each $p,q$ and each $k\geq 2$ the natural map
$E^{p,q}_k(\pi')\to E^{p,q}_k(\pi)$ is a $\QQ$-isomorphism.
\end{lemma}

\begin{proof}
By Lemma \ref{lemma:tensoring-Q-spectral-sequence}, it suffices to prove that
the natural map
$E^{p,q}_k(\pi';\QQ)\to E^{p,q}_k(\pi;\QQ)$ is a $\QQ$-isomorphism assuming the
following properties hold true:
(1) the morphism $(\phi|_F)^*:H^*(F';\QQ) \to H^*(F;\QQ)$ is an isomorphism,
(2) both local systems $\hH(\pi)$ and $\hH(\pi')$ are trivial,
(3) the morphism $\psi^*:H^*(B';\QQ)\to H^*(B;\QQ)$ is an isomorphism.
This is what we are going to prove.

We use ascending induction on $k$. For the case $k=2$ we have an isomorphism
$$E_2^{p,q}(\pi;\QQ)=H^p(B;\hH^q(\pi;\QQ))\simeq H^p(B;\QQ)\otimes H^q(F;\QQ),$$
by the triviality of $\hH(\pi)$ and the universal coefficient theorem.
There is an analogous isomorphism for $\pi'$, and both isomorphisms give
the horizontal maps in the following commutative diagram
$$\xymatrix{E_2^{p,q}(\pi';\QQ)\ar[r]\ar[d] & H^p(B';\QQ)\otimes H^q(F';\QQ) \ar[d]^{\psi^*\otimes(\phi|_F)^*} \\
E_2^{p,q}(\pi;\QQ)\ar[r] & H^p(B;\QQ)\otimes H^q(F;\QQ).}$$
The right vertical arrow is an isomorphism by hypothesis (1) and (3), so the left
vertical arrow is also an isomorphism. This proves the case $k=2$.
The induction step 
follows from the naturality of Serre's spectral sequence.
\end{proof}

\begin{lemma}
\label{lemma:equiv-cohom-Q-trivial}
Let $Y$ be a topological space endowed with an action of a finite group $G$.
Suppose that the induced action of $G$ on $H^*(Y;\QQ)$ is trivial.
Let $Y_G\to BG$ be the Borel construction of the action, and let $\iota:Y\to Y_G$
be the inclusion of one fiber. Then $\iota:H^k(Y_G)\to H^k(Y)$ is a $\QQ$-isomorphism
for every $k$.
\end{lemma}
\begin{proof} Apply the previous lemma to the following diagram:
$$\xymatrix{Y\ar[r]\ar[d] & Y_G\ar[d] \\
\{*\}\ar[r] & BG.}$$
\end{proof}

\subsubsection{Exponent of finite abelian groups}
Let $A$ be a finite abelian group. The {\it exponent} of $A$, which we denote by $\exp A$,
is the lcm of the orders of the elements of $A$. Since $A$ is abelian, $\exp A$ coincides
with $\max\{\ord a\mid a\in A\}$, where $\ord a$ denotes the order of $a$.
If $A$ is a finite abelian group and $B$ is a subquotient of $A$, then clearly $\exp B\leq\exp A$.

If $f:M\to N$ is a morphism of finitely generated abelian groups and $N/f(M)$ is finite, then
we define the {\it exponent} of $f$ as:
$$\exp f:=\exp(N/f(M)).$$
If $N/f(M)$ is infinite, then we set $\exp f:=\infty$.

\begin{lemma}
\label{lemma:basic-properties-exp}
\begin{enumerate}
\item Let $0\to M'\to M\stackrel{\pi}{\longrightarrow} M''\to 0$
be an exact sequence of finite groups. We have
$\exp M\leq(\exp M')(\exp M'').$
\item Let $A\leq B\leq C$ be finitely generated abelian groups.
Suppose that both $\exp B/A$ and $\exp C/B$ are finite. Then
$\exp C/A\leq (\exp B/A)(\exp C/B).$
\end{enumerate}
\end{lemma}
\begin{proof}
We first prove (1). Let $m\in M$. Then $\pi((\exp M'')m)=(\exp M'')\pi(m)=0$, so
$(\exp M'')m\in M'$. Hence, $(\exp M')(\exp M'')m=0$,
which implies $\exp M\leq(\exp M')(\exp M'')$.
Applying (1) to the exact sequence
$0\to B/A\to C/A\to C/B\to 0$ we prove (2).
\end{proof}

\begin{lemma}
\label{lemma:exp-f-vs-exp-f-Z}
Let $f:M\to N$ be a morphism of finitely generated abelian groups. Suppose that $\exp f$ is finite.
Then $\exp(f_{\ZZ})\leq \exp f$.
\end{lemma}
\begin{proof}
Applying the snake lemma to the following commutative diagram with exact rows,
it follows that $N/f(M)$ surjects onto $N_{\ZZ}/f_{\ZZ}(M_{\ZZ})$, from which
the lemma follows immediately:
$$\xymatrix{0\ar[r] & \Tor M\ar[r]\ar[d]_{f|_{\Tor M}} & M\ar[r]\ar[d]_f & M_{\ZZ} \ar[r]\ar[d]_{f_{\ZZ}} & 0 \\
0\ar[r] & \Tor N\ar[r] & N\ar[r] & N_{\ZZ} \ar[r] & 0.}$$
\end{proof}

\begin{lemma}
\label{lemma:exponent-squares}
Consider the following commutative diagram of finitely generated free abelian groups:
$$\xymatrix{A \ar@{^{(}->}[r]^f \ar[d]_g & B \ar[d]^h \\
A' \ar@{^{(}->}[r]^{f'} & B'.}$$
Assume that the inclusions $f,f'$ have finite exponent.
\begin{enumerate}
\item If $h(B)=\lambda B'$ for some nonzero integer $\lambda$, then
\begin{equation}
\label{eq:acotacio-g(A)}
(\exp f)\lambda A'\leq g(A)\leq \frac{\lambda}{\GCD(\lambda,\exp f')}A'.
\end{equation}
\item If $g(A)\leq\mu A'$ for some nonzero integer $\mu$, then
\begin{equation}
\label{eq:acotacio-h(B)}
h(B)\leq\frac{\mu}{\GCD(\mu,\exp f)}B'.
\end{equation}
\item If $h$ has finite exponent then so does $g$, and we have $\exp g\leq (\exp f)(\exp h)$.
\end{enumerate}
\end{lemma}
\begin{proof}
We will use the following claim: if $G'\leq G$ is an inclusion of abelian groups and, for some integers $e,\lambda$, we have $eG'\leq\lambda G$,
then $\GCD(\lambda,e)G'\leq \lambda G$. This follows from writing $\GCD(\lambda,e)$ as $a\lambda+be$,
where $a,b$ are integers.

We prove (1). Let $a'\in A'$. Then $f'(\lambda a')=h(b)$ for some $b\in B$. There exists some $a\in A$ such that $(\exp f)b=f(a)$.
Hence,
$$f'(g(a))=h(f(a))=h((\exp f)b)=(\exp f)h(b)=(\exp f)f'(\lambda a')=f'((\exp f)\lambda a').$$
Since $f'$ is injective, it follows that $g(a)=(\exp f)\lambda a'$. This proves the first inclusion in (\ref{eq:acotacio-g(A)}).
For the other inclusion, note that if $a\in A$ then there exists some $b'\in B'$ such that $f'(g(a))=h(f(a))=\lambda b'$.
This implies that $f'((\exp f')g(a))=\lambda (\exp f')b'$, so there exists some $a'\in A'$ satisfying
$f'((\exp f')g(a))=\lambda f'(a')$. Since $f'$ is injective, we deduce $(\exp f')g(a)=\lambda a'$.
We have thus proved that $(\exp f')g(A)\leq\lambda A'$. By the claim, $\GCD(\lambda,\exp f')g(A)\leq\lambda A'$.
Since this is an inclusion of subgroups of $A'$, and $A'$ is free, we get the second inclusion in (\ref{eq:acotacio-g(A)}).

We now prove (2). We have $(\exp f)h(B)=h((\exp f)B)\leq h(f(A))=f'(g(A))\leq\mu B'$.
By the claim, $\GCD(\mu,\exp f)h(B)\leq\mu B'$. Since $B'$ is free, we get (\ref{eq:acotacio-h(B)}).

We finally prove (3). Since $f'$ is injective, $A'/g(A)$ can be identified with a subgroup of $B'/f'(g(A))=B'/h(f(A))$,
and hence $\exp g=\exp (A'/g(A))\leq \exp (B'/h(f(A)))$.
We have an exact sequence of finite abelian groups
$$0\to\frac{h(B)}{h(f(A))}\to \frac{B'}{h(f(A))}\to \frac{B'}{h(B)}\to 0,$$
so by (1) in Lemma \ref{lemma:basic-properties-exp} we have
$$\exp \frac{B'}{h(f(A))}\leq \exp \frac{h(B)}{h(f(A))} \exp \frac{B'}{h(B)} \leq
\exp \frac{B}{f(A)}\exp \frac{B'}{h(B)} =(\exp f)\cdot (\exp h ).$$
\end{proof}

\subsubsection{Bounds on equivariant cohomology}

\begin{lemma}
\label{lemma:equiv-cohom-exp}
Let $Y$ be a topological space endowed with an action of a finite group $G$. Suppose that $H^*(Y)$ is
finitely generated as an abelian group, and that the induced action of $G$ on $H^*(Y)$ is trivial.
We have:
\begin{equation}
\label{eq:H-G-dins-de-H}
\exp(\iota^*:H^k(Y_G)\to H^k(Y))\leq |G|^k.
\end{equation}
Furthermore, denoting
$C_Y=\max_k\{\exp\Tor H^k(Y)\}$, we have:
\begin{equation}
\label{eq:exp-tor-H-G}
\exp\Tor H^k(Y_G)\leq C_Y|G|^k.
\end{equation}
\end{lemma}
\begin{proof}
Let $\pi:Y_G\to BG$ be the Borel construction for the action of $G$ on $Y$,
and denote $E_r^{p,q}:=E_r^{p,q}(\pi)$. There is a natural
isomorphism $E^{0,k}_2\simeq H^k(Y)$ and a chain of inclusions
$$E^{0,k}_2\supseteq E^{0,k}_3\supseteq\dots\supseteq
E^{0,k}_{k+2}=E^{0,k}_{k+3}=\dots=E^{0,k}_{\infty},$$
and we can identify $\iota^*H^k(Y_G)$ with $E^{0,k}_{\infty}=E^{0,k}_{k+2}$.
For every $p,q$ we have
\begin{equation}
\label{eq:terme-p-q-2}
E^{p,q}_2\simeq H^p(BG,H^q(Y))\simeq H^p(G,H^q(Y)),
\end{equation}
so $\exp E^{p,q}_2\leq |G|$ for every $p>0$ (see e.g. \cite[Chap. III, Corollary 10.2]{Brown}).
If $k\geq 2$ and $p>0$ then $\exp E^{p,q}_k$
is a subquotient of $\exp E^{p,q}_2$, so again $\exp E^{p,q}_k\leq |G|$. Since
$E_r^{0,k}=\Ker(d_{r-1}:E_{r-1}^{0,k}\to E_{r-1}^{r-1,k-r+2})$, we have
$\exp (E_{r-1}^{0,k}/E_r^{0,k})\leq |G|$ for $r>2$. Hence, applying repeatedly (2) in
Lemma \ref{lemma:basic-properties-exp}, we get:
$$\exp(\iota^*:H^k(Y_G)\to H^k(Y))\leq \prod_{r=2}^{k+1} \exp(E_{r}^{0,k}/E_{r+1}^{0,k})\leq|G|^k.$$
This proves (\ref{eq:H-G-dins-de-H}).

We now prove (\ref{eq:exp-tor-H-G}). Since $E^{0,q}_2\simeq H^q(Y)$,
we have $\exp\Tor E^{0,q}_2\leq C_Y$.
If $p>0$ then, by (\ref{eq:terme-p-q-2}), we have
$\exp\Tor E^{p,q}_2=\exp E^{p,q}_2\leq|G|$. Since $E^{p,q}_{\infty}$ is a subquotient of
$E^{p,q}_2$, the same bounds apply to $\exp\Tor E^{p,q}_{\infty}$. There is a filtration
$$H^k(Y_G)=F^0H^k(Y_G)\supseteq F^1H^k(Y_G)
\supseteq\dots\supseteq F^kH^k(Y_G)\supset F^{k+1}H^k(Y_G)=0$$ satisfying $F^jH^k(Y_G)/F^{j+1}H^k(Y_G)\simeq E^{j,k-j}_{\infty}$, so applying
repeatedly (2) in Lemma \ref{lemma:basic-properties-exp}, we get $\exp\Tor H^k(Y_G)\leq C_Y|G|^k$.
\end{proof}

\subsubsection{Complexes with $d_{\QQ}=0$}


The following property will be used in the next two lemmas:
if $f:M\to N$ satisfies $f_{\QQ}=0$, then $f(M)\leq \Tor N$.


\begin{lemma}
\label{lemma:torsion-does-not-increase}
Let $(C^*,d^*)$ be a complex of finitely generated abelian groups.
Suppose that $(d^k)_\QQ=0$ for every $k$. Then, for every $k$, $\Tor H^k(C^*)$ is isomorphic to a subquotient
of $\Tor C^k$, and hence $\exp(\Tor H^k(C^*))\leq \exp(\Tor C^k)$.
\end{lemma}
\begin{proof}
Suppose that $a\in C^k$ satisfies $d^ka=0$ and that $[a]\in \Tor H^k(C^*)$, so, for some $n\in\NN$ and $b\in C^{k-1}$, $na=d^{k-1}b$. Since $(d^{k-1})_{\QQ}=0$, we have $d^{k-1}C^{k-1}\leq\Tor C^k$. Hence $na\in \Tor C^k$, so $a\in\Tor C^k$. It follows that $\Ker (d^k:\Tor C^k\to C^{k+1})$ surjects onto $\Tor H^k(C^*)$.
\end{proof}

\begin{lemma}
\label{lemma:cohomology-d-QQ-zero}
For every complex of finitely generated abelian groups $(C^*,d^*_C)$ satisfying
$(d^*_C)_{\QQ}=0$, and for every $k$, there is a naturally defined inclusion
$\iota_C^k:H^k(C^*)_{\ZZ}\hookrightarrow C^k_{\ZZ}$ satisfying
$\exp \iota_C^k\leq\exp\Tor C^{k+1}$.
If $(C^*,d^*_C)$ and $(D^*,d^*_D)$ are complexes of finitely generated abelian groups satisfying
$(d^*_C)_{\QQ}=0=(d^*_D)_{\QQ}$ and $f^*:C^*\to D^*$ is a morphism of complexes, then for
every $k$ there is a commutative diagram:
$$\xymatrix{H^k(C^*)_{\ZZ}\ar[r]^-{\iota_C^k}\ar[d]_{H^k(f^*)_{\ZZ}} & C^k_{\ZZ}\ar[d]^{f^k} \\
H^k(D^*)_{\ZZ}\ar[r]^-{\iota_D^k}& D^k_{\ZZ}.}$$
\end{lemma}
\begin{proof}
Let $(C^*,d^*_C)$ be a complex of finitely generated abelian groups satisfying $(d^*_C)_{\QQ}=0$.
We have $d^k(C^k)\leq\Tor C^{k+1}$, so the inclusion of $\Tor C^*$ and the projection
to $C^*/\Tor C^*$ give a short exact sequence of complexes
\begin{equation}
\label{eq:tor-C-quot}
0\to (\Tor C^*,d^*_C)\to (C^*,d^*_C)\to (C^*/\Tor^*,0)\to 0.
\end{equation}
The resulting cohomology long exact sequence is:
$$\dots\to H^k(\Tor C^*)\to H^*(C^*)\stackrel{j_C^k}{\longrightarrow} C^k/\Tor C^k\to H^{k+1}(\Tor C^*)\to\dots$$
Applying the functor $(\cdot)_{\ZZ}$ to the map in the middle we get the morphism
$\iota_C^k:=(j_C^k)_{\ZZ}:H^k(C^*)_{\ZZ}\to C^k_{\ZZ}$, and by Lemma \ref{lemma:exp-f-vs-exp-f-Z}
we have
$$\exp\iota_C^k\leq\exp j_C^k\leq\exp H^{k+1}(\Tor C^*)\leq\exp C^{k+1},$$
where the second inequality follows from the exactness.
The second part of the lemma follows from the naturality of the sequence (\ref{eq:tor-C-quot})
and the naturality of the cohomology long exact sequence.
\end{proof}

\subsection{The spectral sequence argument}
\label{ss:conclusion-proof}

The naturality of Serre's spectral sequence applied to the diagram
(\ref{eq:diagrama-fibracions}) gives rise, for every $i$, to a collection of commutative
diagrams of finitely generated abelian groups:
\begin{equation}
\label{eq:Serre-SS-fibracions}
\xymatrix{E_r^{p,q}(\Pi_{3,i}) \ar[r]^{U_i^*}\ar[d]^{d_r^{p,q}(\Pi_{3,i})} &
E_r^{p,q}(\Pi_{2,i}) \ar[r]^{W_i^*}\ar[d]^{d_r^{p,q}(\Pi_{2,i})} &
E_r^{p,q}(\Pi_{1,i}) \ar[d]^{d_r^{p,q}(\Pi_{1,i})}&
E_r^{p,q}(\Pi_0) \ar[l]_-{Z_i^*} \ar[d]^{d_r^{p,q}(\Pi_0)} \\
E_r^{p+r,q-r+1}(\Pi_{3,i}) \ar[r]^{U_i^*} &
E_r^{p+r,q-r+1}(\Pi_{2,i}) \ar[r]^{W_i^*} &
E_r^{p+r,q-r+1}(\Pi_{1,i}) &
E_r^{p+r,q-r+1}(\Pi_0), \ar[l]_-{Z_i^*} }
\end{equation}
where the morphisms $Z_i^*$ are isomorphisms.

\begin{lemma}
\label{lemma:potencia-de-n-i}
For every $i$ and $p,q$, the image of
$U_i^*:E_2^{p,q}(\Pi_{3,i})_{\ZZ}\to E_2^{p,q}(\Pi_{2,i})_{\ZZ}$
is equal to $n_i^p E_2^{p,q}(\Pi_{2,i})_{\ZZ}$.
\end{lemma}
\begin{proof}
The cohomologies of $T_{K_0}/T_{K_0}[n_i]$ and $T_{K_0}$ are torsion free, so
there are natural identifications
$E_2^{p,q}(\Pi_{3,i})\simeq H^p(T_{K_0}/T_{K_0}[n_i];H^q((X_{K,i})_{\Theta_i}))
\simeq H^p(T_{K_0}/T_{K_0}[n_i])\otimes H^q((X_{K,i})_{\Theta_i})$
and
$E_2^{p,q}(\Pi_{2,i})\simeq H^p(T_{K_0};H^q((X_{K,i})_{\Theta_i}))
\simeq H^p(T_{K_0})\otimes H^q((X_{K,i})_{\Theta_i})$.
Via these identifications, the map
$U_i^*:E_2^{p,q}(\Pi_{3,i})_{\ZZ}\to E_2^{p,q}(\Pi_{2,i})_{\ZZ}$
is induced by the morphism
$H^p(T_{K_0}/T_{K_0}[n_i])\to H^p(T_{K_0})$
arising from the projection $T_{K_0}\to T_{K_0}/T_{K_0}[n_i]$.
Since the image of the morphism $H^p(T_{K_0}/T_{K_0}[n_i])\to H^p(T_{K_0})$
is $n_i^pH^p(T_{K_0})$, the lemma follows.
\end{proof}

\begin{lemma}
\label{lemma:exponent-i-torsio-Pi-2-i}
Let $C_X:=\max_k\{\exp\Tor H^k(X_K)\}$.
There exists a constant $C_{\epsilon}$ such that, for
every $i$, we have
\begin{equation}
\label{eq:exp-W-i}
\exp(W_i^*:E^{p,q}_2(\Pi_{2,i})_{\ZZ}\to E^{p,q}_2(\Pi_{1,i})_{\ZZ})\leq C_\epsilon^q
\end{equation}
and
\begin{equation}
\label{eq:exp-tor-E-2-Pi-2-i}
\exp\Tor E^{p,q}_2(\Pi_{2,i})\leq C_XC_\epsilon^q,\qquad \exp\Tor E^{p,q}_2(\Pi_{3,i})\leq C_XC_\epsilon^q.
\end{equation}
\end{lemma}
\begin{proof}
By Lemma \ref{lemma:exp-f-vs-exp-f-Z}, to prove (\ref{eq:exp-W-i})
it suffices to prove the existence of $C_{\epsilon}$ satisfying
$$\exp(W_i^*:E^{p,q}_2(\Pi_{2,i})\to E^{p,q}_2(\Pi_{1,i}))\leq C_{\epsilon}^q$$
for each $i$.
The cohomology of $T_{K_0}$ is torsion free. So we have identifications
$$E_2^{p,q}(\Pi_{2,i})\simeq H^p(T_{K_0};H^q((X_{K,i})_{\Theta_i}))
\simeq H^p(T_{K_0})\otimes H^q((X_{K,i})_{\Theta_i})$$
and
$$E_2^{p,q}(\Pi_{1,i})\simeq H^p(T_{K_0};H^q(X_{K,i}))
\simeq H^p(T_{K_0})\otimes H^q(X_{K,i}).$$
Via these isomorphisms,
$W_i^*:E_2^{p,q}(\Pi_{2,i})_{\ZZ}\to E_2^{p,q}(\Pi_{1,i})_{\ZZ}$
is induced by the identity in $H^p(T_{K_0})$ and the morphism
$H^q((X_{K,i})_{\Theta_i})\to H^q(X_{K,i})$
induced by the inclusion $X_{K,i}\hookrightarrow (X_{K,i})_{\Theta_i}$.

By Lemma \ref{lemma:Ker-w-i-Ker-rho}, we can identify $\Ker w_i$ with $\Ker\rho_{K,i}$.
Since $\Theta_i$ is a subgroup of $\Ker w_i$, it follows that
$|\Theta_i|\leq C_{\epsilon}:=C\cdot C_1$ (recall that $C_1$ was chosen at the beginning
of Section \ref{s:proof-thm:main-rotation-morphism}; see also Subsection
\ref{ss:subgroups-of-rho-K-i-G-i} for the definition of the constant $C$).
With this in mind, (\ref{eq:exp-W-i}) follows from
Lemma \ref{lemma:equiv-cohom-exp}, using the fact that $X_K$ and $X_{K,i}$ are homeomorphic.
Similarly, the inequalities (\ref{eq:exp-tor-E-2-Pi-2-i}) follow from Lemma \ref{lemma:equiv-cohom-exp},
using the fact that the cohomology of $T_{K_0}/T_{K_0}[n_i]$ and $T_{K_0}$ are both torsion free, so that
the torsion of $E_2^{p,q}(\Pi_{j,i})$ ($j=2,3$) is the direct sum of a number of copies of the torsion of
$(X_{K,i})_{\Theta_i}$.
\end{proof}

\begin{lemma}
\label{lemma:isomorfisme-racionals}
The horizontal maps in the diagram obtained
by tensoring (\ref{eq:Serre-SS-fibracions}) by $\QQ$
are all isomorphisms of vector spaces.
\end{lemma}
\begin{proof}
We use Lemma \ref{lemma:isom-Serre-SS-rational} to prove that
$E_r^{p,q}(\Pi_{3,i})_{\QQ}\to E_r^{p,q}(\Pi_{2,i})_{\QQ}$
is an isomorphism: the fibers of the fibrations $\Pi_{3,i}$ and
$\Pi_{2,i}$ are the same, and the map at the
level of bases, $T_{K_0}\to T_{K_0}/T_{K_0}[n_i]$, induces a $\QQ$-isomorphism
in cohomology. To prove that
$E_r^{p,q}(\Pi_{2,i})_{\QQ}\to E_r^{p,q}(\Pi_{1,i})_{\QQ}$
is an isomorphism, we again use Lemma \ref{lemma:isom-Serre-SS-rational}.
The fibrations have the same base, and the induced map between fibers is the
inclusion $X_{K,i}\to (X_{K,i})_{\Theta_i}$, which induces a $\QQ$-isomorphism
in cohomology by Lemma \ref{lemma:equiv-cohom-Q-trivial}. Finally,
$E_r^{p,q}(\Pi_{0})_{\QQ}\to E_r^{p,q}(\Pi_{1,i})_{\QQ}$ is an isomorphism because
$E_r^{p,q}(\Pi_{0})\to E_r^{p,q}(\Pi_{1,i})$ already is.
\end{proof}

An immediate consequence of the previous lemma is the following one.

\begin{lemma}
\label{lemma:the-other-sequences-also-degenerate}
Suppose that for some $k\geq 2$, for every $p,q$, and for every $r$ satisfying $2\leq r\leq k$,
the differentials
$d_r^{p,q}(\Pi_0)_{\QQ}$ are zero. Then, for every $p,q$ and every $r$ satisfying $2\leq r\leq k$, and for any $i$,
the differentials
$d_r^{p,q}(\Pi_{1,i})_{\QQ}$, $d_r^{p,q}(\Pi_{2,i})_{\QQ}$ and $d_r^{p,q}(\Pi_{3,i})_{\QQ}$ are also zero.
\end{lemma}

Another consequence of Lemma \ref{lemma:isomorfisme-racionals} is this.

\begin{lemma}
\label{lemma:horizontal-maps-inclusions}
Suppose that for some $k\geq 2$, for every $p,q$, and for every $r$ satisfying $2\leq r\leq k$,
the differentials
$d_r^{p,q}(\Pi_0)_{\QQ}$ are zero. Then, for every $p,q$ and every $r$ satisfying
$2\leq r\leq k$, and for any $i$,
when we apply  the functor $(\cdot)\mapsto (\cdot)_{\ZZ}$ to the diagram (\ref{eq:Serre-SS-fibracions}),
all horizontal maps become inclusions.
\end{lemma}

\begin{lemma}
\label{lemma:inclusions-pagines}
Suppose that for some $k\geq 2$, for every $p,q$, and for every $r$ satisfying $2\leq r<k$,
the differentials $d_r^{p,q}(\Pi_0)_{\QQ}$ are zero.
Then, for every $p,q$ and for any $i$,
there is a commutative diagram in which the vertical maps are inclusions,
and the horizontal maps are the natural
ones induced by the maps in diagram (\ref{eq:diagrama-fibracions}):
\begin{equation}
\label{eq:inclusions-pagines}
\xymatrix{E_k^{p,q}(\Pi_{3,i})_{\ZZ} \ar@{^{(}->}[r]^{U_i^*}\ar@{_{(}->}[d]^{\iota_k^{p,q}(\Pi_{3,i})} &
E_k^{p,q}(\Pi_{2,i})_{\ZZ} \ar@{^{(}->}[r]^{W_i^*}\ar@{_{(}->}[d]^{\iota_k^{p,q}(\Pi_{2,i})} &
E_k^{p,q}(\Pi_{1,i})_{\ZZ} \ar@{_{(}->}[d]^{\iota_k^{p,q}(\Pi_{1,i})}
\\
E_2^{p,q}(\Pi_{3,i})_{\ZZ} \ar@{^{(}->}[r]^{U_i^*} &
E_2^{p,q}(\Pi_{2,i})_{\ZZ} \ar@{^{(}->}[r]^{W_i^*} &
E_2^{p,q}(\Pi_{1,i})_{\ZZ}.
}
\end{equation}
%
Furthermore, for $j=1,2,3$,
\begin{equation}
\label{eq:cota-superior-exponent-iota}
\exp \iota_k^{p,q}(\Pi_{j,i})\leq C_X^{k-2}C^{q(q-1)/2},
\end{equation}
where $C_X,C_{\epsilon}$ are the numbers given by Lemma \ref{lemma:exponent-i-torsio-Pi-2-i}.
\end{lemma}
\begin{proof}
By Lemma \ref{lemma:the-other-sequences-also-degenerate}, for every $p,q$, and for every $r$ satisfying
$2\leq r<k$, the differentials
$d_r^{p,q}(\Pi_{1,i})_{\QQ}$, $d_r^{p,q}(\Pi_{2,i})_{\QQ}$ and $d_r^{p,q}(\Pi_{3,i})_{\QQ}$ are zero.
Choose some $r$ satisfying $2\leq r<k$.
Since, for every $i,j$, $E^{*,*}(\Pi_{j,i})_{r+1}=H^*(E^{*,*}_r(\Pi_{j,i}),d_r)$, applying
Lemma \ref{lemma:cohomology-d-QQ-zero} we obtain a commutative diagram
\begin{equation}
\label{eq:bricks}
\xymatrix{E_{r+1}^{p,q}(\Pi_{3,i})_{\ZZ} \ar@{^{(}->}[r]\ar@{_{(}->}[d] &
E_{r+1}^{p,q}(\Pi_{2,i})_{\ZZ} \ar@{^{(}->}[r]\ar@{_{(}->}[d] &
E_{r+1}^{p,q}(\Pi_{1,i})_{\ZZ} \ar@{_{(}->}[d]
\\
E_r^{p,q}(\Pi_{3,i})_{\ZZ} \ar@{^{(}->}[r] &
E_r^{p,q}(\Pi_{2,i})_{\ZZ} \ar@{^{(}->}[r] &
E_r^{p,q}(\Pi_{1,i})_{\ZZ},
}
\end{equation}
and the exponent of each vertical map $E_{r+1}^{p,q}(\Pi_{j,i})_{\ZZ} \to E_{r}^{p,q}(\Pi_{j,i})_{\ZZ}$
is bounded above by $\exp\Tor E_r^{p+r,q-r+1}(\Pi_{j,i})$.
Applying recursively Lemma \ref{lemma:torsion-does-not-increase} we have
$$\exp\Tor E_r^{p+r,q-r+1}(\Pi_{j,i})\leq \exp\Tor E_2^{p+r,q-r+1}(\Pi_{j,i}).$$
By Lemma \ref{lemma:exponent-i-torsio-Pi-2-i} we have
$$\exp\Tor E_2^{p+r,q-r+1}(\Pi_{j,i})\leq C_XC_{\epsilon}^{q-r+1}$$
if $j=2,3$ (this only makes sense if $q-r+1\geq 0$, but otherwise $E_2^{p+r,q-r+1}(\Pi_{j,i})=0$,
so $E_{r+1}^{p,q}(\Pi_{j,i})_{\ZZ}=E_{r}^{p,q}(\Pi_{j,i})_{\ZZ}$ in this case).
Similar arguments prove that
$$\exp\Tor E_2^{p+r,q-r+1}(\Pi_{1,i})\leq C_X.$$
Stacking the diagrams (\ref{eq:bricks}) for $r=2,\dots,k-1$ (from bottom to top) we obtain the diagram
(\ref{eq:inclusions-pagines}), and the bound for $\exp \iota_k^{p,q}(\Pi_{j,i})$
follows from the previous bounds and (2) in Lemma \ref{lemma:basic-properties-exp}.
\end{proof}

\begin{theorem}
\label{thm:degeneracio-ssSerre}
The cohomology Serre spectral sequence for the fibration $\Pi_0$ over the rationals
degenerates at the second page.
\end{theorem}
\begin{proof}
We want to prove that $d^{p,q}_k(\Pi_0;\QQ)=0$ for every $k\geq 2$ and any $p,q$.
By Lemma \ref{lemma:tensoring-Q-spectral-sequence}, this is equivalent to the vanishing of
$d^{p,q}_k(\Pi_0)_{\QQ}$ for every $k\geq 2$ and any $p,q$.
We are going to prove this using induction on $k$.

Suppose that, for some $p,q$, $d_2^{p,q}(\Pi_0)_{\QQ}\neq 0$. Then
$d_2^{p,q}(\Pi_0)_{\ZZ}\neq 0$ as well. So there is some $a\in E^{p,q}_2(\Pi_0)_{\ZZ}$
such that $d_2^{p,q}(a)\neq 0$. Let $\Lambda$ be the biggest integer such that
$d_2^{p,q}(a)\in\Lambda E^{p+2,q-1}_2(\Pi_0)_{\ZZ}$. Choose $i$ big enough so that
$$\frac{n_i^2}{C_{\epsilon}^q}>\Lambda.$$
The following commutative diagram is obtained applying the functor $(\cdot)\mapsto (\cdot)_{\ZZ}$
to the diagram (\ref{eq:Serre-SS-fibracions}) for $k=2$:
\begin{equation}
\label{eq:Serre-SS-fibracions-pagina-2}
\xymatrix{E_2^{p,q}(\Pi_{3,i})_{\ZZ} \ar@{^{(}->}[r]^{U_i^*}\ar[d]^{d_2^{p,q}(\Pi_{3,i})_{\ZZ}} &
E_2^{p,q}(\Pi_{2,i})_{\ZZ} \ar@{^{(}->}[r]^{W_i^*}\ar[d]^{d_2^{p,q}(\Pi_{2,i})_{\ZZ}} &
E_2^{p,q}(\Pi_{1,i})_{\ZZ} \ar[d]^{d_2^{p,q}(\Pi_{1,i})_{\ZZ}}&
E_2^{p,q}(\Pi_0)_{\ZZ} \ar[l]_-{Z_i^*} \ar[d]^{d_2^{p,q}(\Pi_0)_{\ZZ}} \\
E_2^{p+2,q-1}(\Pi_{3,i})_{\ZZ} \ar@{^{(}->}[r]^{U_i^*} &
E_2^{p+2,q-1}(\Pi_{2,i})_{\ZZ} \ar@{^{(}->}[r]^{W_i^*} &
E_2^{p+2,q-1}(\Pi_{1,i})_{\ZZ} &
E_2^{p+2,q-1}(\Pi_0)_{\ZZ}. \ar[l]_-{Z_i^*} }
\end{equation}
Let $a_i'=Z_i^*(a)\in E^{p,q}_2(\Pi_{1,i})_{\ZZ}$.
Since the maps $Z_i^*$ are isomorphisms, it follows that for every $\Lambda'>\Lambda$ we have
\begin{equation}
\label{eq:d-no-pot-ser-massa-gran}
d_2^{p,q}(\Pi_{1,i})_{\ZZ}(a_i')\notin \Lambda' E^{p+2,q-1}_2(\Pi_{1,i})_{\ZZ}.
\end{equation}

We claim that $d_2^{p,q}(\Pi_{2,i})_{\ZZ}(E_2^{p,q}(\Pi_{2,i})_{\ZZ})\leq n_i^2 E_2^{p+2,q-1}(\Pi_{2,i})_{\ZZ}$.
Indeed, if $b\in E_2^{p,q}(\Pi_{2,i})_{\ZZ}$, then by Lemma \ref{lemma:potencia-de-n-i}
$n_i^pb$ is equal to $U_i^*(c)$ for some $c\in E_2^{p,q}(\Pi_{3,i})_{\ZZ}$. Then, by the commutativity
of the first square in the diagram above, and using again Lemma \ref{lemma:potencia-de-n-i},
we have
$$d_2^{p,q}(\Pi_{2,i})_{\ZZ}(n_i^pb)=d_2^{p,q}(\Pi_{2,i})_{\ZZ}(U_i^*c)=U_i^*d_2^{p,q}(\Pi_{3,i})_{\ZZ}(c)\in n_i^{p+2}E_2^{p+2,q-1}(\Pi_{2,i})_{\ZZ},$$
which implies the claim.

Combining (\ref{eq:exp-W-i}) with (2) in Lemma \ref{lemma:exponent-squares}, it follows that
\begin{equation}
\label{eq:d-massa-gran}
d_2^{p,q}(\Pi_{1,i})_{\ZZ}(E_2^{p,q}(\Pi_{1,i})_{\ZZ})\leq \frac{n_i^2}{\GCD(n_i^2,C_{\epsilon}^q)}E_2^{p+2,q-1}(\Pi_{1,i})_{\ZZ}.
\end{equation}
Since, by our choice of $i$,
$$\frac{n_i^2}{\GCD(n_i^2,C_{\epsilon}^q)}\geq \frac{n_i^2}{C_{\epsilon}^q}>\Lambda,$$
(\ref{eq:d-massa-gran}) contradicts (\ref{eq:d-no-pot-ser-massa-gran}). We have thus proved that $d_2^{p,q}(\Pi_0)_{\QQ}=0$ for every
$p,q$.

Now, for the inductive step, suppose that $k>2$ and that, for every $r$ satisfying $2\leq r<k$,
$d_r^{p,q}(\Pi_0)_{\QQ}=0$ for every $p,q$. Arguing as in the case $k=2$,
in order to prove that $d_k^{p,q}(\Pi_0)_{\QQ}=0$, it suffices to prove, for any choice of $p,q$,
the existence of a sequence of integers $\Lambda_i$ satisfying $\Lambda_i\to\infty$ such that
$$d_k^{p,q}(\Pi_{1,i})_{\ZZ}(E^{p,q}_k(\Pi_{1,i})_{\ZZ})\leq \Lambda_i E^{p+k,q-k+1}_k(\Pi_{1,i})_{\ZZ}.$$

Applying (3) in Lemma \ref{lemma:exponent-squares} to the right square in the diagram in the statement
of Lemma \ref{lemma:inclusions-pagines}, together with the bound on $\exp \iota^{p,q}_k(\Pi_{2,i})$ and
the bound (\ref{eq:exp-W-i}) in Lemma \ref{lemma:exponent-i-torsio-Pi-2-i}, we conclude that
\begin{equation}
\label{eq:cota-W-i-pagina-k}
\exp(W_i^*:E^{p,q}_k(\Pi_{2,i})_{\ZZ} \to E^{p,q}_k(\Pi_{1,i})_{\ZZ})\leq C_{\epsilon}^qC_X^{k-2}C^{q(q-1)/2}.
\end{equation}
Similarly, applying (1) in Lemma \ref{lemma:exponent-squares} to the left square we conclude:
\begin{equation}
\label{eq:forquilla-U-i}
\exp(\iota_k^{p,q}(\Pi_{3,i})_{\ZZ})n_i^pE_k^{p,q}(\Pi_{2,i})\leq U_i^*(E_k^{p,q}(\Pi_{3,i})_{\ZZ})\leq
\frac{n_i^p}{\GCD(n_i^p,\exp(\iota_k^{p,q}(\Pi_{2,i})_{\ZZ}))}E_k^{p,q}(\Pi_{2,i})_{\ZZ}.
\end{equation}
To conclude the proof we will apply the previous inequalities to the following commutative diagram:
\begin{equation}
\label{eq:Serre-SS-fibracions-pagina-k}
\xymatrix{E_k^{p,q}(\Pi_{3,i})_{\ZZ} \ar@{^{(}->}[r]^{U_i^*}\ar[d]^{d_k^{p,q}(\Pi_{3,i})_{\ZZ}} &
E_k^{p,q}(\Pi_{2,i})_{\ZZ} \ar@{^{(}->}[r]^{W_i^*}\ar[d]^{d_k^{p,q}(\Pi_{2,i})_{\ZZ}} &
E_k^{p,q}(\Pi_{1,i})_{\ZZ} \ar[d]^{d_k^{p,q}(\Pi_{1,i})_{\ZZ}}
\\
E_k^{p+k,q-k+1}(\Pi_{3,i})_{\ZZ} \ar@{^{(}->}[r]^{U_i^*} &
E_k^{p+k,q-k+1}(\Pi_{2,i})_{\ZZ} \ar@{^{(}->}[r]^{W_i^*} &
E_k^{p+k,q-k+1}(\Pi_{1,i})_{\ZZ}
}
\end{equation}
Applying both inequalities in (\ref{eq:forquilla-U-i}) to the left square in the diagram, and arguing as
in the case $k=2$, we conclude that for every $b\in E_k^{p,q}(\Pi_{2,i})_{\ZZ}$
we have:
$$d_k^{p,q}(\Pi_{2,i})_{\ZZ}(\exp(\iota_k^{p,q}(\Pi_{3,i})_{\ZZ})n_i^pb)\in
\frac{n_i^{p+k}}{\GCD(n_i^{p+k},\exp(\iota_k^{p+k,q-k+1}(\Pi_{2,i})_{\ZZ}))}E^{p+k,q-k+1}_k(\Pi_{2,i})_{\ZZ}.$$
Using the claim at the beginning of the proof of Lemma \ref{lemma:exponent-squares}, the previous
formula implies that, defining
$$\lambda_i:=\frac{n_i^{p+k}}{\GCD(n_i^{p+k},\exp(\iota_k^{p+k,q-k+1}(\Pi_{2,i})_{\ZZ}))}
\qquad
\text{and}
\qquad
\mu_i:=\frac{\lambda_i}{\GCD(\lambda_i,\exp(\iota_k^{p,q}(\Pi_{3,i})_{\ZZ})n_i^p)},$$
we have
$$d_k^{p,q}(\Pi_{2,i})_{\ZZ}(E_k^{p,q}(\Pi_{2,i})_{\ZZ})\leq \mu_iE_k^{p+k,q-k+1}(\Pi_{2,i})_{\ZZ}.$$
Finally, applying (2) in Lemma \ref{lemma:exponent-squares} to the right square in the diagram above,
using the previous inclusion, we conclude that,
setting
$$\Lambda_i:=\frac{\mu_i}{\GCD(\mu_i,\exp(W_i^*:E^{p,q}_k(\Pi_{2,i})_{\ZZ} \to E^{p,q}_k(\Pi_{1,i})_{\ZZ}))},$$
we have
$$d_k^{p,q}(\Pi_{1,i})_{\ZZ}(E_k^{p,q}(\Pi_{1,i})_{\ZZ})\leq \Lambda_i E_k^{p+k,q-k+1}(\Pi_{1,i})_{\ZZ}.$$
It only remains to bound $\Lambda_i$. In the following inequalities we denote by $\const(q,k)$ a constant,
varying from place to place, depending only on $X,q,k,C,C_X,C_{\epsilon}$, but {\it independent of $i$}:
\begin{align*}
\Lambda_i &\geq \frac{\mu_i}{\exp(W_i^*:E^{p,q}_k(\Pi_{2,i})_{\ZZ} \to E^{p,q}_k(\Pi_{1,i})_{\ZZ})}
\geq \frac{\mu_i}{\const(q,k)} \qquad\text{by (\ref{eq:cota-W-i-pagina-k})} \\
&\geq \frac{\lambda_i}{\exp(\iota_k^{p,q}(\Pi_{3,i})_{\ZZ})n_i^p \const(q,k)} \\
&\geq \frac{n_i^{p+k}}{\exp(\iota_k^{p+k,q-k+1}(\Pi_{2,i})_{\ZZ})\exp(\iota_k^{p,q}(\Pi_{3,i})_{\ZZ})n_i^p \const(q,k)}
\geq \frac{n_i^{k}}{\const(q,k)}\to\infty
\end{align*}
as $i\to\infty$, where in the last line we have used (\ref{eq:cota-superior-exponent-iota}).
With this estimate, the proof of Theorem \ref{thm:degeneracio-ssSerre} is complete.
\end{proof}

\subsection{Conclusion of the proof of Theorem \ref{thm:main-rotation-morphism}}
\label{ss:conclusion-thm:main-rotation-morphism}
We have now proved that $K$ satisfies each of the properties from (P1) to (P4), and hence
that $K$ belongs to $\kK(X)$.
Let us recall, for the reader's benefit, where each of the properties
has been established.
(Note that the definition of $\kK(X)$ was given at the beginning of Section \ref{s:proof-thm:main-rotation-morphism}.)
Property (P1), the fact that $H^*(X_K)$ is finitely generated as a
$\ZZ$-module, has been proved in Lemma \ref{lemma:finite-generated-cohomology-over-integers}.
Property (P2), the existence of a subgroup $K_0^*\leq K^*$ of finite index whose induced action
on $H^*(X_K)$ is trivial, was proved in Subsection \ref{ss:def-K-0}, right after the statement of Theorem \ref{thm:trivial-module-structure}.
Property (P3), the degeneracy over the rationals of the Serre spectral of the fibration $X_K\times_{K_0^*}V_K\to V_K/K_0^*$, was proved in Theorem \ref{thm:degeneracio-ssSerre}.
Finally, property (P4), the injectivity of $\psi_K^*:H^*(T_K)\to H^*(X)$, follows, by
the discussion in Subsection \ref{ss:preparing-ss}, from (P3).

By assumption, $K$ also belongs to the set $\lL$ defined in (\ref{eq:def-kK}),
so $|\rho_i(\Gamma_i)\cap\Ker r_K|$ is bounded as $i\to\infty$.
Hence,
$$C(G_i,X)\leq |\Ker (r_K:\rho_i(G_i)\to T_K)|$$
is bounded as $i\to\infty$. We have thus obtained a
contradiction with (\ref{eq:C-G-i-X-diverges}), so the proof of
Theorem \ref{thm:main-rotation-morphism} is now complete.

\section{Proof of Theorem \ref{thm:trivial-module-structure}}
\label{s:proof-thm:trivial-module-structure}

The main ingredient in the proof of Theorem \ref{thm:trivial-module-structure} will be the following.

\begin{theorem}
\label{thm:roots-of-integer-matrices}
Let $A\in \GL(m,\ZZ)$, and suppose that there exists a sequence of integers $r_i\to\infty$
and matrices $B_i\in\GL(m,\ZZ)$ satisfying $B_i^{r_i}=A$. Then there exists a nonzero integer
$e$ such that $A^e=\Id$.
\end{theorem}
\begin{proof}
Let $\lambda_1,\dots,\lambda_m\in\CC$
(resp. $\lambda_{i,1},\dots,\lambda_{i,m}\in\CC$) be the eigenvalues
of $A$ (resp. $B_i$) repeated with multiplicity. Up to reordering
the numbers $\lambda_{i,1},\dots,\lambda_{i,m}$ we may assume that
$\lambda_{i,j}^{r_i}=\lambda_j$ for every $i,j$. Hence, for every $j$,
$|\lambda_{i,j}|=|\lambda_j|^{1/r_i}\to 1$ as $i\to\infty$, so
the numbers $\lambda_{i,j}$ stay in a bounded subset of $\CC$.
It follows that the coefficients of the polynomial
$P_i(x)=\prod_j(\lambda_{i,j}-x)$ stay in a bounded set as $i\to\infty$. But
$P_i(x)$ is the characteristic polynomial of $B_i$, and hence its coefficients
are integers. So there is a finite set of polynomials in $\ZZ[x]$ which contains
$P_i(x)$ for every $i$. It follows that there exists a polynomial $P(x)\in\ZZ[x]$
and sequence of integers $i_k\to\infty$ such that $P_{i_k}(x)=P(x)$ for every $k$.
So if we denote by $\zeta_1,\dots,\zeta_m\in\CC$ the roots of $P(x)$ repeated
with multiplicity, for each $k$ there is a permutation $\sigma_k$ of the set
$\{1,\dots,m\}$ such that $\lambda_{i_k,j}=\zeta_{\sigma_k(j)}$. We may choose
$k\neq k'$ such that $\sigma_k=\sigma_{k'}$, which implies that
$\lambda_{i_k,j}=\lambda_{i_{k'},j}$ for every $j$. It then follows that, for every $j$,
$$\lambda_j^{r_{i_k}}=\lambda_{i_{k'},j}^{r_{i_{k'}}r_{i_k}}=
\lambda_{i_k,j}^{r_{i_{k'}}r_{i_k}}=\lambda_j^{r_{i_{k'}}},$$
so setting $e_0=r_{i_{k'}}-r_{i_k}$ we have $\lambda_j^{e_0}=1$, and hence
$\lambda_j$ is a root of unity.
It follows that all roots of $P$ are roots of unity as well, and consequently there exists
some integer $e$ such that $\lambda_{i_k,j}^e=1$ for every $k$ and $j$. Then $B_{i_k}^e$
is unipotent, and so is $A^e=(B_{i_k}^e)^{r_{i_k}}$.

Let $C=A^e-\Id$ and $D_k=B_{i_k}^e-\Id$.
If $C=0$ then $A^e=\Id$ and we are done. So suppose from now on that $C\neq 0$.
We will see that this leads to a contradiction.

To simplify the notation we define $s_k=r_{i_k}$,
so that $s_k\to\infty$ and $\Id+C=(\Id+D_k)^{s_k}$. The matrices $C$ and $D_k$
are nilpotent of size $m\times m$, so they satisfy $C^m=D_k^m=0$.
Hence,
\begin{equation}
\label{eq:C-funcio-D}
C=\left(s_k\atop 1\right)D_k+\left(s_k\atop 2\right)D_k^2+\dots+\left(s_k\atop m-1\right)D_k^{m-1}.
\end{equation}
We claim that for every $1\leq g\leq m-1$ the matrix
$D_k^g$ is a rational linear combination of $C^g,C^{g+1},\dots,C^{m-1}$. To prove the claim
we use descending induction on $g$, beginning with $g=m-1$. Raising both sides of the equality
(\ref{eq:C-funcio-D}) to $m-1$ and using $C^m=D^m_k=0$ we conclude:
$$C^{m-1}=\left(s_k\atop 1\right)^{m-1}D_k^{m-1},$$
which proves the case $g=m-1$ of the claim. Now suppose that $g<m-1$, and that the claim
is true for bigger values of $g$. Raising both sides of the equality
(\ref{eq:C-funcio-D}) to $g$, and using the induction hypothesis as well as $C^{m-1}=0$ we obtain:
\begin{align*}
C^g &=\left(\left(s_k\atop 1\right)D_k+\left(s_k\atop 2\right)D_k^2+\dots+\left(s_k\atop m-1\right)D_k^{m-1}\right)^g \\
&=\left(s_k\atop 1\right)^gD_k^g+(\text{rational linear combination of $D_k^{g+1},\dots,D_k^{m-1}$}) \\
&=\left(s_k\atop 1\right)^gD_k^g+(\text{rational linear combination of $C_k^{g+1},\dots,C_k^{m-1}$}).
\end{align*}
Rearranging terms this implies the claim for the present value of $g$, concluding the proof of the induction step.
We have thus proved (taking $g=1$) the existence for every $k$ of rational numbers $\xi_{k,1},\dots,\xi_{k,m-1}$ such that
\begin{equation}
\label{eq:D-funcio-C}
D_k=\xi_{k,1}C+\xi_{k,2}C^2+\dots+\xi_{k,m-1}C^{m-1}.
\end{equation}
Now define the vector spaces
$$V_i=\Ker(C^i:\QQ^m\to\QQ^m),\qquad W_{k,i}=\Ker(D_k^i:\QQ^m\to\QQ^m).$$
For every $i$ and $k$,
raising both sides of the equality (\ref{eq:C-funcio-D}) (resp. (\ref{eq:D-funcio-C})) to $i$
we obtain the inclusion $V_i\subseteq W_{k,i}$ (resp. $W_{k,i}\subseteq V_i$).
Hence $V_i=W_{k,i}$ for every $i$ and $k$.

There are inclusions
$$V_1\subseteq V_2\subseteq\dots\subseteq V_m=\QQ^m,
\qquad
W_{k,1}\subseteq W_{k,2}\subseteq \dots\subseteq W_{k,m}=\QQ^m$$
for every $k$. If we had $V_1=V_2$ then we would have $V_1=V_i$ for every $i$ (use ascending induction on $i$ and
the equality $V_i=\{v\in\QQ^m\mid Cv\in V_{i-1}\}$), which, taking $i=m$,
would imply that $V_1=\QQ^m$, and hence $C=0$, contradicting our assumption. Hence the inclusion
$V_1\subset V_2$ is strict.

Now let $v$ be any element in $V_2\setminus V_1$. Multiplying $v$ by an integer we may assume
that $v\in\ZZ^m\setminus\{0\}$. Then $Cv\in\ZZ^m$. We have $D_k^2v=0$ for every $k$ because $v\in V_2=W_{k,2}$, and hence
(\ref{eq:C-funcio-D}) implies that $Cv=s_kD_kv$ for every $k$. Since $D_kv\in\ZZ^m$ and $s_k\in\ZZ$ for every $k$,
and $s_k\to\infty$, we necessarily
have $Cv=0$. But then $v\in V_1$, contradicting our choice of $v$. We have thus reached a contradiction
assuming that $C\neq 0$, and consequently the proof of the theorem is now complete.
\end{proof}

We are now ready to prove Theorem \ref{thm:trivial-module-structure}.
Let $M$ be a finitely generated $\ZZ$-module and let $z:M\to M$ be an automorphism.
Suppose that there is a sequence of integers $r_i\to\infty$ and automorphisms $w_i:M\to M$
satisfying $w_i^{r_i}=z$. We want to prove that $z$ has finite order in $\Aut M$.

Since $M$ is finitely generated, its torsion $\Tor M$ is finite. Like any automorphism
of $M$, $z$ preserves $\Tor M$, so it induces an automorphism $\ov{z}$ of $M/\Tor M$.
Picking an isomorphism $\phi:M/\Tor M\stackrel{\simeq}{\longrightarrow}\ZZ^m$
we may identify $\ov{z}$ with a matrix
in $\GL(m,\ZZ)$. If $w$ is any automorphism of $M$ satisfying $w^r=z$ and we denote
by $\ov{w}$ the automorphism of $M/\Tor M$ induced by $w$, then we have $\ov{w}^r=\ov{z}$.
So if we denote by $B\in\GL(m,\ZZ)$ the matrix corresponding to $\ov{w}$ via $\phi$
then we have $B^r=A$. Consequently, by Theorem \ref{thm:roots-of-integer-matrices} the assumption
of the corollary implies the existence of a nonzero integer $e$
such that $\ov{z}^e$ is the identity on $M/\Tor M$.
Since $\Tor M$ is finite, replacing $e$ by some of its nonzero multiples
we may assume that $z^e$ restricts to the identity on $\Tor M$ and that it induces the
identity on $M/\Tor M$. Hence, there is a morphism $\psi:M\to \Tor M$ such that
for every $x\in M$ we have $z^e(x)=x+\psi(x)$. Since $z^e$ acts trivially on $\Tor M$
we have $z^e(\psi(x))=\psi(x)$, so
for any integer $h$ we have $z^{he}(x)=x+h\psi(x)$. Consequently, if $h$ denotes the cardinal of $\Tor M$,
$z^{he}$ is the identity on $M$.

\section{Proof of Theorem \ref{thm:fixed-points-kernel-rotation-morphism}}

\subsection{Passing to integral cohomology}
\label{s:integral-cohomology}
Let $X$ be an $n$-dimensional WLS manifold, and let $\Omega\in H^2(X;\RR)$ be a WLS class.
The aim of this subsection is to prove that there exists an
{\it integral} WLS class on $X$.
Given a commutative ring $A$ with unit and an integer $d$,
we denote
$$V_{A,d}:=\bigoplus_{j\geq 0}\Lambda^{d-2j}H^1(X;A).$$
Let $c:V_{A,d}\to H^*(X;A)$ denote the cup product map.
For any $\omega\in H^2(X;A)$, let
$$c_{\omega,A,d}:V_{A,d}\to H^d(X;A)$$
be the linear map defined by the condition that $c_{\omega,A,d}(\tau)=c(\tau)\smile\omega^j$
for each $\tau\in \Lambda^{d-2j}H^1(X;A)$.
Define
$$\uU_d=\{\omega\in H^2(X;\RR)\mid c_{\omega,\RR,d}\text{ is surjective}\}.$$
We have $\Omega\in\uU_{n-1}\cap\uU_n$, so $\uU_{n-1}\cap\uU_n$ is nonempty.

The maps $c_{\omega,\RR,d}$ depend continuously on $\omega$. Since being surjective is an open condition
on continuous families of linear maps between finite dimensional real vector spaces, the sets
$\uU_d$ are open. It follows that $\uU_{n-1}\cap\uU_n$ is a nonempty open subset of $H^2(X;\RR)$, so
$\vV:=\uU_{n-1}\cap\uU_n\cap H^2(X;\QQ)$ is nonempty. Let $\mu$ be an element of $\vV$.
Pick a nonzero integer $r$ such that $r\mu$ belongs to the image of the natural map
$H^2(X)\exh H^2(X)/\Tor H^2(X)\hookrightarrow H^2(X;\QQ)$. Clearly, $r\mu\in\vV$.

Let $\lambda\in H^2(X)$ be any lift of $r\mu$. By construction, $\lambda$ is an integral WLS class on $X$.
Define
$$\delta_d:=|\Coker (c_{\lambda,\ZZ,d}:V_{\ZZ,d}\to H^d(X)/\Tor H^d(X))|\in\NN\cup\{\infty\}.$$
Since $\lambda$ is a lift of an element of $\vV$, both $\delta_{n-1}$ and $\delta_n$ are finite natural numbers.

\subsection{Lifting finite cyclic group actions to circle bundles}

In this section we prove some results on lifting finite cyclic group actions to line bundles.
We give an ad hoc self-contained exposition, which
is roughly as long as it would be to merely set up the notation and refer to existing general results
(see e.g. \cite[\S 6]{Mu2018}).


We denote the homotopy class
of a map $h:X\to S^1$ by $[h]$, and we identify it with $h^*\theta\in H^1(X)$,
where $\theta\in H^1(S^1)$ is the generator specified in \S\ref{sss:modules-and-tori}.

Let $\pi:P\to X$ be a circle bundle. Any map $h:X\to S^1$ defines a gauge transformation $\mu_h:P\to P$
sending $p\in P$ to $p\cdot h(\pi(p))$.
Conversely, any gauge transformation $\mu:P\to P$ defines a map $h_{\mu}:X\to S^1$ by the condition
that $\mu(p)=p\cdot h_{\mu}(\pi(p))$.
We denote $[\mu]:=[h_{\mu}]\in H^1(X)$.

Let $\gamma:X\to X$ be a homeomorphism of order $r$ inducing the trivial
automorphism of $H^1(X)$ and such that $P$ and $\gamma^*P$ are isomorphic circle bundles.
For any bundle automorphism $\kappa:P\to P$ lifting $\gamma$,
the $r$-th composition $\kappa^r:P\to P$ lifts the identity on $X$, so we may define
$$c^r(\gamma,P):=[\kappa^r]\otimes 1\in H^1(X)\otimes_{\ZZ} \ZZ/r.$$
The following will be used below: for any $h:X\to S^1$ and any bundle automorphism
$\kappa:P\to P$ lifting $\gamma$ we have $\kappa^{-1}\mu_h\kappa=\mu_{h\circ \gamma}$.
Since $\gamma$ acts trivially on $H^1(X)$, we have
$[\kappa^{-1}\mu_h\kappa]=[\mu_{h\circ \gamma}]=[h\circ \gamma]=[h].$

\begin{lemma}
\label{lemma:independence-of-lift}
The element $c^r(\gamma,P)$ does not depend on the choice of $\kappa$.
\end{lemma}
\begin{proof}
Suppose that $\kappa':P\to P$ is another lift of $\gamma$. We then have $\kappa'=\kappa\cdot h=\kappa\mu_h $ for some
$h:X\to S^1$. Then we compute:
\begin{align*}
[(\kappa')^r] &= [(\kappa\mu_h)^r]=[\kappa^r (\kappa^{-(r-1)}\mu_h\kappa^{r-1})\dots (\kappa^{-2}\mu_h\kappa^2)(\kappa{-1}\mu_h\kappa)\mu_h] \\
&=[\kappa^r]+[\kappa^{-(r-1)}\mu_h\kappa^{r-1}]+\dots+[\kappa^{-2}\mu_h\kappa^2]+[\kappa{-1}\mu_h\kappa]+[\mu_h] \\
&=[\kappa^r]+r[h],
\end{align*}
which proves that $[(\kappa')^r]\otimes 1=[\kappa^r]\otimes 1$ as elements of $H^1(X)\otimes_{\ZZ} \ZZ/r$.
\end{proof}

\begin{lemma}
\label{lemma:existence-lifting-line-bundle}
We have $c^r(\gamma,P)=0$ if and only if $\gamma$ can be lifted to a principal bundle map
$\kappa:P\to P$ satisfying $\kappa^r=\Id_P$.
\end{lemma}
\begin{proof}
If $\gamma$ admits a lift $\kappa:P\to P$ of order $r$ then $[\kappa^r]=0$, so $c^r(\gamma,P)=0$.
Conversely, suppose that $c^r(\gamma,P)=0$. Then there exists some lift $\kappa:P\to P$ of $\gamma$
and some function $h:X\to S^1$ such that $\kappa^r=(\mu_h)^r=\mu_{rh}$. Denote $j=rh$. We have
$\mu_j=\kappa^r=\kappa^{-1}\kappa^r\kappa=\kappa^{-1}\mu_j\kappa=\mu_{j\circ \gamma}$,
which implies that $j=j\circ \gamma$. As a
consequence, $r(h\circ \gamma-h)=0$ in $\RR/\ZZ$, which implies that $h\circ \gamma-h=\xi$
for some $\xi\in S^1$ such that $r\xi=0$.
A computation similar to that in the proof of Lemma \ref{lemma:independence-of-lift}
gives
$$(\kappa\mu_h^{-1})^r=-\frac{(r-1)+\dots+2+1}{2}\xi=-\frac{r(r-1)}{2}\xi.$$
If $r$ is odd this implies that $(\kappa\mu_h^{-1})^r=0$, so $\kappa':=\kappa\mu_h^{-1}$ is a lift of
$\gamma$ of order $r$ and we are done. If $r$ is even
then either $(\kappa\mu_h^{-1})^r=0$ or $(\kappa\mu_h^{-1})^r=1/2$ (here we abusively denote elements in
$\RR$ and their class in $\RR/\ZZ$ with the same symbol). In the first case
we are done as before. In the second case we have
$(\kappa\mu_h^{-1}\mu_{1/(2r)})^r=0$, so we are also done.
\end{proof}

The proof of the following lemma is immediate.

\begin{lemma}
\label{lemma:lifting-c-f-P}
Let $\pi:X'\to X$ be a continuous map of topological manifolds, and let $\gamma':X'\to X'$ be a homeomorphism
of order $r$ satisfying $\pi\circ \gamma'=\gamma\circ\pi$. We have $(\gamma')^*\pi^*P\simeq \pi^*P$ and
$c^r(\gamma',\pi^*P)=\pi^*c^r(\gamma,P)$.
\end{lemma}

\subsection{From the cohomology of a fiber to that of a fibration}
\label{s:from-fiber-to-total-space}

Let $B$ be a closed, connected and oriented $s$-dimensional manifold.
Let $p\in B$ be any point.
Suppose given maps
$$F\stackrel{\iota}{\hookrightarrow} E\stackrel{\pi}{\longrightarrow} B,$$
where $\pi$ is a locally trivial fibration and $\iota$ is the inclusion of the fiber at $p$.
We assume $F$ to be a general topological space.

We use the notation introduced at the beginning of Section \ref{ss:ignoring-torsion}.
For every $r\geq 2$,
all differentials $d_r^{s,*}(\pi):E_r^{s,*}(\pi)\to E_r^{s+r,*-(r-1)}(\pi)$ are zero, because
the cohomology of $B$ vanishes in dimensions bigger than $s$. Hence,
$E_2^{s,*}(\pi)$ surjects onto $E_{\infty}^{s,*}(\pi)$.
Define a linear map
$$\iota_!:H^*(F)\to H^{s+*}(E)$$
as the composition of the morphisms
\begin{multline*}
H^*(F)=\hH^*(\pi)_p\exh (\hH^*(\pi)_p)_{\pi_1(B,p)}=H_0(B;\hH^*(\pi))\stackrel{\PD}{\longrightarrow} \\
\stackrel{\PD}{\longrightarrow} H^s(B;\hH^*(\pi))=E_2^{s,*}(\pi)\exh E_{\infty}^{s,*}(\pi)\hookrightarrow H^{s+*}(E),
\end{multline*}
where $(\hH^*(\pi)_p)_{\pi_1(B,p)}$ denotes the coinvariants for the action of $\pi_1(B,p)$
on $\hH^*(\pi)_p$, and $\PD$ denotes Poincar\'e duality.
We remark that if $F$ is a closed, connected and oriented manifold,
then $\iota_!$ is the usual pushforward or {\it umkehr} map,
defined as conjugation of $\iota_*:H_*(F)\to H_*(E)$ with Poincar\'e dualities on $F$ and $E$.

The naturality of the Serre spectral sequence implies the following lemma.

\begin{lemma}
\label{lemma:naturality-T}
Suppose given a commutative diagram
$$\xymatrix{F\ar@{^(->}[r]^{\iota} \ar[d]^{\psi} & E \ar[r]^{\pi}\ar[d]^{\Psi} & B \ar@{=}[d] \\
F'\ar@{^(->}[r]^{\iota'} & E' \ar[r]^{\pi'} & B ,}$$
where $\pi':E'\to B$ is a locally trivial fibration and $\iota'$ is the inclusion of the fiber at $p$.
For any class $\beta\in H^*(F')$ we have $\Psi^*(\iota'_!\beta)=\iota_!\psi^*\beta$.
\end{lemma}

\begin{lemma}
\label{lemma:pullback-T}
Let $[B]\in H^s(B)$ the Poincar\'e dual of the class of a point.
For any $\eta\in H^*(E)$ we have $\iota_!\iota^*\eta=\pi^*[B]\smile\eta$.
\end{lemma}
\begin{proof}
Let $1_F\in H^0(F)$ denote the cohomology class of the augmentation map.
The definition of $\iota_!$ implies that $\iota_!1_F=\pi^*[B]$.
For any $\eta\in H^*(E)$, the class $\iota^*\eta$ sits inside $E^{0,*}_2=\hH^*(\pi)_p^{\pi_1(B,p)}\subset
\hH^*(\pi)_p=H^*(F)$, and it survives in all pages of the spectral
sequence (here $\hH^*(\pi)_p^{\pi_1(B,p)}$ denotes
the invariants of the action of $\pi_1(B,p)$). The lemma follows from the previous observations and
the multiplicativity of Serre spectral sequence.
\end{proof}

\subsection{Finding points with big stabilizer}

In this subsection we prove Theorem \ref{thm:fixed-points-kernel-rotation-morphism}.
Let $P\to X$ be a principal circle bundle whose Euler class
$e(P)\in H^2(X)$ is equal to $\lambda$, the class chosen at the end of Subsection \ref{s:integral-cohomology}.
Let us define
$$C_3:=\max\{\delta_{n-1},\delta_n\}.$$

Suppose given an effective and $H^*$-trivial action on $X$ of $G:=\ZZ/p^e$.
Suppose that the rotation morphism $\rho:G\to T_H$ is trivial.
Let $g\in G$ be a generator and let $\gamma:X\to X$ be the homeomorphism
given by the action of $g$.
Since $G$ acts trivially on $H^2(X)\simeq\{\text{circle bundles on $X$}\}/\simeq$, we have
$\gamma^*P\simeq P$, so the class $c^{p^e}(\gamma,P)\in H^1(X)\otimes\ZZ/p^e$ is well defined.
Suppose that
$$p^f=\max\{|G_x|\mid x\in X\}.$$

We distinguish two cases.

\noindent {\bf Case 1.} Suppose first that $c^{p^e}(\gamma,P)=0$. Then the action of $G$ on $X$ lifts to an action on
$P$. Let $R<S^1$ be the group of $p^f$-th roots of unity, and let
$Q:=P/R$. Then $Q$ is a principal bundle over $S^1/R\simeq S^1$, and its Euler class satisfies $e(Q)=p^fe(P)=p^f\lambda$.
The action of $G$ on $P$ induces an action on $Q$, and for
every $y\in X$ the naturally induced action of $G_y$ on $Q_y$ is trivial, because $G_y\leq R$. Hence, if we
denote by $\zeta:X\to X':=X/G$ be the quotient map, there is a principal $S^1$-bundle
$Q'$ such that $Q\simeq\zeta^*Q'$. This implies that $p^f\lambda\in\zeta^*H^2(X')$.

By the definition of $\delta_n$, there exist classes
$\alpha_1,\dots,\alpha_r\in H^1(X)$ such that
$$\alpha_1\smile\dots\smile\alpha_r\smile\lambda^k=\delta_n[X],$$
where $k\geq 0$ is an integer satisfying $r+2k=n$, and where $[X]\in H^n(X)$ is a generator.

By Corollary \ref{cor:trivial-rotation-pullback-H-1}, the morphism
$\zeta^*:H^1(X')\to H^1(X)$ is surjective, because the rotation morphism of the action
of $G$ on $X$ is trivial. Since $p^f\lambda\in\zeta^*H^2(X')$, it follows that
$$p^{fk}\delta_n[X]=
\alpha_1\smile\dots\smile\alpha_r\smile (p^f\lambda)^k\in\zeta^*H^n(X').$$

Since $G$ acts effectively on $X$, by \cite[Lemma 4.4]{Mu2021} we have $\zeta^*H^n(X')\leq p^e H^n(X)$.
Consequently, $p^e$ divides $\delta_np^{fk}$, so $\delta_np^{fk}\geq p^e$.
It follows that, if $x\in X$ is any point satisfying $|G_x|=p^f$,
$$|G|=p^e\leq \delta_np^{fk}
=\delta_n|G_x|^{k}
\leq\delta_n|G_x|^{n/2}
\leq C_3|G_x|^{n/2},$$
where the second inequality follows from the equality $r+2k=n$.
Hence Theorem \ref{thm:fixed-points-kernel-rotation-morphism} is proved in this case.

\noindent{\bf Case 2.} Now suppose that $c^{p^e}(f,P)\neq 0$.
Choose a lift of $c^{p^e}(f,P)\in H^1(X)\otimes_{\ZZ}\ZZ/p^e$
of the form $p^s\alpha\in H^1(X)$, where $\alpha\in H^1(X)$ is primitive
(i.e., $\alpha$ is not a nontrivial multiple of an element in $H^1(X)$).

By Poincar\'e duality, there exists
some class $\eta\in H^{n-1}(X)$ such that $\alpha\smile\eta=[X]$,
where $[X]\in H^n(X)$ is a generator.
By the definition of $\delta_{n-1}$, the class of $\delta_{n-1}\eta$
in $H^{n-1}(X)/\Tor H^{n-1}(X)$ belongs to the image of $c_{\lambda,\ZZ,n-1}$,
so we can write
$$\delta_{n-1}[X]=\alpha\smile \delta_{n-1}\eta=\alpha\smile\eta_0+\dots+\alpha\smile\eta_{n-1},$$
where
$$\eta_j=\beta_{j,1}\smile\dots\smile\beta_{j,n-1-2j}\smile\lambda^j,$$
for classes $\beta_{j,i}\in H^1(X)$. Let $p^m$ be the biggest power of $p$ dividing $\delta_{n-1}$.
Take some $j$ such that $p^{m+1}$ does not divide $\alpha\smile\eta_j$,
and define $s:=n-1-2j$, $t:=j$, and $\beta_i:=\beta_{j,i}$.
Since $p^{m+1}$ does not divide
$\alpha\smile\beta_1\smile\dots\smile\beta_r\smile\lambda^t$,
necessarily $\alpha\smile\beta_1\smile\dots\smile\beta_r$ is
nonzero, so the classes $\alpha,\beta_1,\dots,\beta_r$ are linearly independent.
Let $K\leq H$ be the subgroup generated by $\alpha,\beta_1,\dots,\beta_r$.
Since the rotation morphism of the action of $G$ on $X$ is trivial, by
Theorem \ref{thm:rotation-morphism} there is a map $\phi:X\to T_H$ satisfying:
\begin{enumerate}
\item[(P1)] $\phi$ represents the canonical element in $[X,T_H]$, and
\item[(P2)] the map $\phi$ is $G$-invariant, i.e., $\phi(g\cdot x)=\phi(x)$ for every $g\in G$ and $x\in X$.
\end{enumerate}
Denote by $r_K:T_H\to T_K$ the restriction map, and let $\phi_K:=r_K\circ\phi$.
By (P1), the image of $\phi_K^*:H^1(T_K)\to H^1(X)$ is $K$.
Choosing a suitable orientation of $T_K$,
the Poincar\'e dual of the point class in $H_0(T_K)$ is an element $[T_K]\in H^{r+1}(T_K)$ satisfying
\begin{equation}
\label{eq:phi-star-T-K}
\phi_K^*[T_K]=\alpha\smile\beta_1\smile\dots\smile\beta_r.
\end{equation}
Denote $V_K=\Hom(K,\RR)$. The natural projection
$\pi_K:V_K\to T_K$ is a universal cover.
Consider the pullback through $\phi_K$ of the universal cover of $\pi_K$,
$$\ov{X}:=\{(x,u)\in X\times V_K\mid \phi_K(x)=\pi_K(u)\},$$
and let $\zeta:\ov{X}\to X$ denote the projection. Since $\phi_K^*H^1(T_K)=K$ and $\alpha\in K$, we have
\begin{equation}
\label{eq:zeta-pullback-alpha}
\zeta^*\alpha=0.
\end{equation}
By (P2), the action
of $G$ on $X$, combined with the trivial action on $V_K$, defines an action on $X\times V_K$
which leaves $\ov{X}$ invariant. In other words, the action of $G$ on $X$ lifts to an action
on $\ov{X}$. The action of $K^*=\Hom(K,\ZZ)$ on $V_K$ through addition allows to identify
$T_K$ with $V_K/K^*$, so $\zeta:\ov{X}\to X$ is a principal $K^*$-bundle. The action of
$K^*$ on $\ov{X}$ commutes with the action of $G$.

Consider the quotient maps
$$\zeta:X\to X':=X/G,\qquad \ov{\zeta}:\ov{X}\to \ov{X}':=\ov{X}/G.$$
The spaces and maps defined so far fit in the following commutative diagram:
$$\xymatrix{\ov{X}\ar@{^(->}[r]^-{\iota} \ar@/^2pc/[rrr]<3pt>^{\xi}
\ar[dd]_{\ov{\zeta}}
& \ov{X}\times_{K^*}V_K\ar[rr]^-{\Theta} \ar[dd]_Z \ar[rd]^{\Pi} & & X \ar[dd]^{\zeta}\ar[ld]^{\phi_K} \\
&& T_K \\
\ov{X}'\ar@{^(->}[r]^-{\iota'} \ar@/_2pc/[rrr]^{\xi'}
& \ov{X}'\times_{K^*}V_K\ar[rr]^-{\Theta'} \ar[ru]^{\Pi'} & & X',\ar[lu]_{\phi_K'}}$$
where $\iota,\iota'$ are the inclusions of the fibers of $\Pi,\Pi'$ over the class of $0$ in $T_K$,
and $Z$ is induced by the map $\ov{\zeta}\times\Id_{V_K}:\ov{X}\times V_K\to \ov{X}'\times V_K$.
The projections $\Theta$ and $\Theta'$ are homotopy equivalences, because they are both fibrations
with fiber $V_K$, which is contractible.

Consider the pullback circle bundle $\ov{P}:=\xi^*P\to \ov{X}$.
We have
\begin{multline}
\label{eq:push-forward-iota-lambda}
\iota_!(e(\ov{P})^t)=\iota_!\iota^*\Theta^*e(P)^t=\iota_!\iota^*\Theta^*\lambda^t=
\Pi^*[T_K]\smile\Theta^*\lambda^t=
\\
=\Theta^*(\phi_K^*[T_K]\smile\lambda^t)=\Theta^*(\alpha\smile\beta_1\smile\dots\smile\beta_r\smile\lambda^t),
\end{multline}
where in the third equality we have used Lemma \ref{lemma:pullback-T}

By Lemma \ref{lemma:lifting-c-f-P} 
and (\ref{eq:zeta-pullback-alpha}), the action of $G$ on $\ov{X}$
lifts to an action on $\ov{P}$. Let $R<S^1$ be the group of $p^f$-th roots of unity, and let
$\ov{Q}:=\ov{P}/R$. Then $\ov{Q}$ is a principal bundle over $S^1/R\simeq S^1$,
and its Euler class satisfies $e(\ov{Q})=p^fe(\ov{P})=p^f\iota^*\Theta^*\lambda$.
Arguing as in Case 1, we conclude that there exists a circle bundle $\ov{Q}'\to \ov{X}'$
such that $\ov{\zeta}^*\ov{Q}'=\ov{Q}$. By Lemma \ref{lemma:naturality-T},
$$p^{ft}\iota_!e(\ov{P})^t=\iota_!e(\ov{Q})^t=\iota_!\ov{\zeta}^*e(\ov{Q}')^t=Z^*\iota'_!e(\ov{Q}')^t$$
Combining this with (\ref{eq:push-forward-iota-lambda}) we conclude that
$$p^{ft}\Theta^*(\alpha\smile\beta_1\smile\dots\smile\beta_r\smile\lambda^t)
\in Z^*H^n(\ov{X}'\times_{K^*}V_K).$$
Since both $\Theta$ and $\Theta'$ are homotopy equivalences, it follows that
$$p^{ft}\alpha\smile\beta_1\smile\dots\smile\beta_r\smile\lambda^t\in \zeta^*H^n(X').$$
As in Case 1, \cite[Lemma 4.4]{Mu2021} implies $\zeta^*H^n(X')\leq p^e H^n(X)$,
because $G$ acts effectively on $X$. Since $\alpha\smile\beta_1\smile\dots\smile\beta_r\smile\lambda^t\in H^n(X)\simeq\ZZ$ is not
divisible by $p^{m+1}$, $p^e$ divides $p^mp^{ft}$, so, letting $x\in X$ be any point such that
$|G_x|=p^f$, we have $|G|=p^e\leq p^mp^{ft}=p^m|G_x|^t\leq \delta_{n-1}|G_x|^t$.
Since $t\leq n/2$, it follows that
$$|G|\leq \delta_{n-1}|G_x|^{n/2}\leq C_3|G_x|^{n/2},$$
and hence Theorem \ref{thm:fixed-points-kernel-rotation-morphism} is proved in Case 2.

\section{Proofs of the main theorems}
\label{s:proofs-main-thms}

\subsection{Proof of Theorem \ref{thm:toral-rank-Carlsson}}
\label{ss:proof-thm:toral-rank-Carlsson}
Let $X$ be an $n$-dimensional WLS manifold. Let $C_M$ be the number assigned to $M$
by Lemma \ref{lemma:Minkowski}.

By \cite[Theorem 2.5]{MannSu} there is a natural number $\mu$ such that, for every prime $p$
and every effective action of $(\ZZ/p)^m$ on $X$, we have $m\leq\mu$.

Choose a natural number $C$. Let $C_3$ be the number given by
applying Theorem \ref{thm:fixed-points-kernel-rotation-morphism} to $X$, let
$C_1:=(C^{[(n+1)/2]}C_3)^{\mu}$,
and let $C_2=C_2(X,C_1)$ be the number associated to $X$ and $C_1$
by Theorem \ref{thm:main-rotation-morphism}. Using the constant $C_M$ given by Lemma \ref{lemma:Minkowski}, we define
$$C':=1+C_M!C_2!.$$
We are going to prove that Theorem \ref{thm:toral-rank-Carlsson} holds for this choice
of $C'$. Let $p^e\geq C'$ be a prime power, and suppose given an
action of $G:=(\ZZ/p^e)^k$ on $X$ all of whose stabilizers have size
at most $C$. By the definition of $C_M$, there is a subgroup $G_1\leq G$
satisfying $[G:G_1]\leq C_M$ and whose action on $H^*(X)$ is trivial.
By \cite[Lemma 2.1]{Mu2021}, there is a subgroup $G_2\leq G_1$ isomorphic to
$(\ZZ/p^{e_2})^k$ and satisfying $C_M!p^{e_2}\geq p^e$.

Let $H:=H^1(X)$ and let
$\rho:G_2\to T_H$
be the rotation morphism of the action. Let $G_3:=\Ker\rho$.
We claim that $|G_3|\leq C_1$.
Indeed, we may write $G_3\simeq \ZZ/p^{d_1}\times\dots\times \ZZ/p^{d_m}$,
where $d_1\leq\dots\leq d_m$ and $m\leq\mu$.
Let $G_3'\leq G_3$ be the subgroup
corresponding to the last factor $\ZZ/p^{d_m}$.
It suffices to prove that $|G_3'|\leq C_1^{1/\mu}=C^{[(n+1)/2]}C_3$.
But if $|G_3'|>C^{[(n+1)/2]}C_3$ then, by Theorem \ref{thm:fixed-points-kernel-rotation-morphism},
there would be a point $x\in X$ such that
$$|G_x|\geq|(G_3')_x|\geq (C_3^{-1}|G_3'|)^{2/n}>C,$$
contradicting our assumption, so the claim is proved.

By Theorem \ref{thm:main-rotation-morphism} there exists a subgroup $K\leq H$,
such that the composition $\rho_K:=r_K\circ\rho:G_2\to T_K$ satisfies
$$|\Ker \rho_K|\leq C_2,$$
and the map $\psi_K:=r_K\circ\psi:X\to T_K$ has the property that
$$\psi_K:H^*(T_K)\to H^*(X)$$
is injective. We need to prove that $\dim T_K\geq k$.

The bound $|\Ker \rho_K|\leq C_2$ and \cite[Lemma 2.3]{Mu2021} implies the existence of
a subgroup $G_4\leq\rho(G_2)$ isomorphic to $(\ZZ/p^{e_4})^k$, where
$C_2!p^{e_4}\geq p^{e_2}$.
It follows that
$$p^{e_4}\geq \frac{p^e}{C_M!C_2!}>1,$$
and consequently, by \cite[Lemma 2.5]{Mu2021}, $\dim T_K\geq k$, as we wanted to prove.

\subsection{Proof of Theorem \ref{thm:main-big-prime-p}}
\label{ss:proof-thm:main-big-prime-p}

Let $X$ be a WLS manifold.
Let $C_M$ be the number given by applying Lemma \ref{lemma:Minkowski} to $X$.
Let $C_2$ be the number assigned to $C_1:=1$ and $X$ by Theorem \ref{thm:main-rotation-morphism},
and let $C'$ be the number given by applying Theorem \ref{thm:fixed-points-kernel-rotation-morphism} to $X$.

Following \cite{Mu2022}, we say that an action of a group $G$ on $X$ has the {\it weak fixed point
property} if for every $g\in G$ there exists some $x\in X$ such that $g\cdot x=x$.
By \cite[Theorem 1.6]{Mu2022} there is a constant $C_F$ such that, for any action of
a finite group $G$ on $X$ with the weak fixed point property, there is a subgroup $G'\leq G$
such that $[G:G']\leq C_F$ and such that $X^{G'}\neq\emptyset$.

We are going to prove that the number
$$C:=1+\max\{C_2,C',C_F,C_M\}$$
has the desired properties. Let $p\geq C$ be a prime, and assume that $G:=(\ZZ/p)^r$
acts effectively on $X$. Since $C>C_M$, the action of $G$ on
$H^*(X)$ is trivial. Let $H:=H^1(X)$, and let $\rho:G\to T_H$ be the rotation morphism.
Let $G_0:=\Ker\rho$, and choose a subgroup $G_1\leq G$ such that $G=G_0\times G_1$.
Then $\rho$ maps isomorphically $G_1$ to $\rho(G_1)<T_H$.
By Theorem \ref{thm:main-rotation-morphism} there is:
\begin{enumerate}
\item[(Q1)] a subgroup $K\leq H$ such that $\rho_K:=r_K\circ\rho:G_1\to T_K$ is injective
and such that $H^*(X_K)$ is a finitely generated $\ZZ$-module, where $r_K:T_H\to T_K$ is the
restriction map, and
\item[(Q2)] a finite index subgroup $K_0^*\leq K^*$ such that Serre's spectral sequence for
rational cohomology of the fibration $X_K\times_{K_0^*}V_K\to V_K/K_0^*$
degenerates at the second page.
\end{enumerate}
Let $\phi:X\to T_H$ be a $G$-equivariant morphism as given by Theorem \ref{thm:rotation-morphism}.
Let $k=\dim K\otimes \QQ$, choose an isomorphism $\zeta:T_K\simeq T^k$, and let
$\xi:X\to T^k$ be the composition $\zeta\circ r_K\circ \phi$.
By (2) in Theorem \ref{thm:rotation-morphism}, $\xi$ is $G_1$-equivariant
with respect to the morphism $\zeta\circ\rho:G_1\to T^k$, so
statement (1) in Theorem \ref{thm:main-big-prime-p} is proved.

Let $\pi:X'\to X$ be the pullback of the covering
$\pi_T:T_{K_0}:=V_K/K_0^*\to V_K/K^*=T_K$ under the
map $\psi_K:X\to T_K$.
Arguing as in Lemma \ref{lemma:homotopy-commutative-diagram},
there is a homotopy commutative diagram
$$\xymatrix{X_K\times_{K_0^*}V_K \ar[r]^-{\simeq} \ar[d] &  X' \ar[d]^{\zeta\circ\psi_K\circ\pi} \\
T_{K_0}\ar[r]^{\zeta\circ\pi_T} & T^k.}$$
The morphism $\pi_T^*:H^*(T_K;\QQ)\to H^*(T_{K_0};\QQ)$ is an isomorphism,
so, denoting $X'':=X_K$, the property in (Q2) above implies statement
statement (2) in Theorem \ref{thm:main-big-prime-p}.

In thus remains to prove statement (3) in Theorem \ref{thm:main-big-prime-p}.
By Theorem \ref{thm:fixed-points-kernel-rotation-morphism}, any $g\in G_0$
has a fixed point, so $G_0$ has the weak fixed point property.
By \cite[Theorem 1.6]{Mu2022} and our choice of $C$, it follows that $X^{G_0}\neq\emptyset$.

Let $F\subseteq X^{G_0}$ be any connected component.
By Smith theory (see \cite[Chap V, Theorem 2.2]{Bo}), $F$ a $\ZZ/p$-cohomology manifold.
To finish, we want to prove that $\dim F\geq k$. We will use an argument similar
to that in \cite[Lemma 9.3]{Mu2021}.
Let $\pi_K:X_K\to X$ be the projection.
By (2) in Theorem \ref{thm:rotation-morphism},
the action of $G_0$ on $X$ lifts to an action on $X_K$ which commutes with
the action of $K^*\simeq \ZZ^k$, and we have
$X_K^{G_0}=\pi_K^{-1}(X^{G_0})$. Let $F_K:=\pi_K^{-1}(F)$.
Then $F_K$ is the union of some of the connected components of $X_K^{G_0}$.
By Smith theory (see e.g. \cite[Lemma 5.1]{Mu2018}), $\dim H^*(X_K^{G_0};\ZZ/p)\leq
\dim H^*(X_K;\ZZ/p)<\infty$, so $X_K^{G_0}$, and hence $F_K$, has finitely many connected
components.

Let $F_K'\subseteq F_K$ be one connected component.
The action of $K^*$ on $X_K$ permutes the connected components of $F_K$,
so the subgroup $K^*_F\leq K^*$ consisting of elements leaving $F_K'$ invariant satisfies $[K^*:K^*_F]<\infty$, which implies $K^*_F\simeq\ZZ^k$.

Since $\dim H^*(F_K';\ZZ/p)<\infty$, there is a finite index subgroup $K^*_{F,1}\leq K^*_F$
which acts trivially on $H^*(F_K';\ZZ/p)$. Again, $K^*_{F,1}\simeq\ZZ^k$.
The quotient $F_K'/K^*_{F,1}$ is homotopy equivalent to $F_{K'}\times_{K^*_{F,1}}V_K$.
Let $d=\max\{r\mid H^r(F_K';\ZZ/p)\neq 0\}$. Then the entry $E_2^{k,d}$
in the Serre spectral sequence with $\ZZ/p$-coefficients for the fibration
$F_{K'}\times_{K^*_{F,1}}V_K\to V_K/K^*_{F,1}\simeq T^k$
is nonzero (because $K^*_{F,1}$ acts trivially on $H^d(F_K';\ZZ/p)\neq 0$, and
$H^k(V_K/K^*_{F,1};\ZZ/p)\simeq\ZZ/p$), and $E_2^{k,d}$
maps isomorphically onto $E_{\infty}^{k,d}$, because in every page the differentials
to and from $E_r^{k,d}$ are zero. This implies that
$H^{k+d}(F_K'/K^*_{F,1};\ZZ/p)\neq 0$. Since $F_K'/K^*_{F,1}$ is a finite cover of $F$,
it follows that $\dim F\geq k+d\geq k$, as we wished to prove.

\subsection{Proof of Theorem \ref{thm:main-disc-deg-sym}}

Let $X$ be a WLS $n$-dimensional manifold. Suppose there is a natural number $m$,
a sequence of natural numbers $r_i\to\infty$, and an effective action of
$(\ZZ/r_i)^m$ on $X$ for each $i$.
Let $\pP=\{p\text{ prime}\mid p\text{ divides some $r_i$}\}$. We distinguish two cases.

\noindent{\bf Case 1.}
If the set $\pP$
is infinite, we may choose a sequence of primes $p_i$ in $\pP$ satisfying
$p_i\to\infty$. For each $i$ there is some $j$ such that $p_i$ divides
$r_j$, and consequently $(\ZZ/p_i)^m$ is isomorphic to a subgroup of
$(\ZZ/r_{j})^m$. Therefore, restricting the action of $(\ZZ/r_j)^m$ to such subgroup,
we get an effective action of $(\ZZ/p_i)^m$ on $X$.
Let $C$ be the constant assigned to $X$ by Theorem \ref{thm:main-big-prime-p}.
Suppose that $p_i\geq \max\{3,C\}$. By Theorem \ref{thm:main-big-prime-p},
there is a splitting $(\ZZ/p_i)^m\simeq (\ZZ/p_i)^{m-s}\times(\ZZ/p_i)^s$
such that:
\begin{enumerate}
\item the action of $(\ZZ/p_i)^{m-s}$ on $X$ has nonempty fixed point set $F$, and
each connected component of $F$ is a $\ZZ/p_i$-cohomology manifold of dimension $\geq s$,
\item there is a continuous map $X\to T^k$, for some $k\geq s$, inducing an injective
map $H^*(T^k)\to H^*(X)$.
\end{enumerate}

\begin{lemma}
\label{lemma:embedding-in-GL(n-s)}
$(\ZZ/p_i)^{m-s}$ is isomorphic to a subgroup of $\GL(n-s,\RR)$.
\end{lemma}
\begin{proof}
Let $G:=(\ZZ/p_i)^{m-s}$ and let $x\in X^G$. For each subgroup $K\leq G$, let
$\nu(K)$ denote the dimension at $x$ of the $\ZZ/p$-cohomology manifold $X^K$
(see \cite[Chap V, Theorem 2.2]{Bo}). The arguments in the proof of
\cite[Theorem 3.2]{Mu2022} prove that the function
assigning to each subgroup $K\leq G$ the number $\nu(K)$ satisfies the Borel--Smith conditions.
(For the definition of the Borel--Smith conditions, see \cite[Chap III, Definition (5.1)]{tD} or \cite[Lemma 3.1]{Mu2022}). Of course, we have $\nu(\{1\})=n$.
It then follows from a result of Dotzel and Hamrick
(see  \cite[Section 1]{DH} or \cite[Chap III, Theorem 5.13]{tD}) that there is a linear
representation $\xi:G\to\GL(n,\RR)$ such that, for every subgroup $K\leq G$,
we have $\nu(K)=\dim (\RR^n)^K$.
Since $G$ is finite, we may assume, replacing $\xi$ by a conjugate if necessary, that
$\xi(G)\leq \O(n,\RR)$.
Since $G$ acts effectively on $X$, we have $\nu(K)<n$ for every
nontrivial $\{1\}\neq K\leq G$ (see \cite[Chap I, Corollary 4.6]{Bo}),
so the representation $\xi$ is faithful.
Since by (1) each connected component of $X^G$ is a $\ZZ/p_i$-cohomology manifold of dimension $\geq s$, it follows that $\nu(G)\geq s$.
This implies that $V:=(\RR^n)^G$ satisfies $\dim V\geq s$.
Now, the definition of $V$ implies that
$$\xi(G)\leq \{A\in\O(n,\RR)\mid A|_V=\Id_V\}\simeq\O(V^{\perp}).$$
Since $\dim V\geq s$, $\dim V^{\perp}=n-\dim V\leq n-s$, so $\O(V^{\perp})$,
and consequently $\xi(G)\simeq G$, is isomorphic to a subgroup of $\O(n-s,\RR)<\GL(n-s,\RR)$.
\end{proof}

\begin{lemma}
\label{lemma:cota-m-s}
Let $r$ be equal either to an odd prime or to $4$.
If there exists a faithful representation $\xi:(\ZZ/r)^t\to\GL(u,\RR)$ for some natural numbers $t,u$,
then $t\leq u/2$.
\end{lemma}
\begin{proof}
Suppose first that $r$ is equal to and odd prime $p$.
Let $G:=(\ZZ/p)^t$. Since $G$ is finite abelian,
the complex representation $\xi\otimes\CC:G\to\GL(u,\CC)$
splits as a direct sum of irreducible complex $1$-dimensional characters.
Since $p$ is odd, the only
$1$-dimensional complex representation of $G$ that is isomorphic to its conjugate is the trivial representation. Consequently, we may write
$\xi\otimes\CC\simeq \eta_1\oplus\ov{\eta_1}\oplus\dots\oplus
\eta_k\oplus\ov{\eta_k}\oplus {\mathbf 1}^l$, where $\eta_i:G\to \CC^*$ is nontrivial for
every $i$ and ${\mathbf 1}:G\to\CC^*$ denotes the trivial $1$-dimensional representation.
Clearly, $2k+l=u$. Since $\xi$ is injective, so is $\eta_1\oplus\ov{\eta_1}\oplus\dots\oplus
\eta_k\oplus\ov{\eta_k}\oplus {\mathbf 1}^l$. The latter factors through the representation
$\eta_1\oplus\dots\oplus\eta_k$, so $\eta_1\oplus\dots\oplus\eta_k$ is necessarily injective. Since $\Ker\eta_i\simeq(\ZZ/p)^{t-1}$ for every $i$ (as each $\eta_i$ is $1$-dimensional), in order for the equality $\bigcap_i\Ker\eta_i=\{0\}$ to hold true we must have $k\geq t$. It follows that
$2t\leq 2k\leq u$, so $t\leq u/2$.

Now suppose that $r=4$, and let $G:=(\ZZ/4)^t$. In this case,
if $\eta:G\to\CC^*$ is a character isomorphic to its conjugate,
then $\eta$ is not necessarily trivial, but nevertheless $\eta|_{2G}$ is trivial (we use here additive notation on $G$), because $\eta$ takes values in $\{1,-1\}$. Consequently, we may write
$$\xi\otimes\CC\simeq \eta_1\oplus\ov{\eta_1}\oplus\dots\oplus
\eta_k\oplus\ov{\eta_k}\oplus\nu_1\oplus\dots\oplus\nu_v\oplus{\mathbf 1}^l,$$
where $\nu_j\simeq\ov{\nu_j}$ for every $j$. Restricting to $2G$, we obtain
$$(\xi\otimes\CC)|_{2G}\simeq
\eta_1|_{2G}\oplus\ov{\eta_1}|_{2G}\oplus\dots\oplus
\eta_k|_{2G}\oplus\ov{\eta_k}|_{2G}\oplus{\mathbf 1}^{l+v}.$$
Arguing at this point as in the case of odd primes, with $G$ replaced by $2G\simeq(\ZZ/2)^t$, we obtain again $t\leq u/2$.
\end{proof}

The existence of an injective
map $H^*(T^k)\to H^*(X)$ given in item (2) above implies that $s\leq k\leq\tau(X)$.
Combining Lemmas \ref{lemma:embedding-in-GL(n-s)} and \ref{lemma:cota-m-s} we obtain $m-s\leq (n-s)/2$. Combining the two inequalities we get:
$$m=(m-s)+s\leq \frac{n-s}{2}+s=\frac{n+s}{2}\leq\frac{n+\tau(X)}{2}.$$
This proves the first part of Theorem \ref{thm:main-disc-deg-sym} when
$\pP$ is infinite. As for the second part, note that if $m=n$
then $\tau(X)\geq n$, so $\tau(X)=n$ and hence $X$ is rationally hypertoral in the sense of \cite{Mu2021}. Hence,
applying \cite[Theorem 1.3(2)]{Mu2021}, the proof of Theorem \ref{thm:main-disc-deg-sym}
is concluded in the present case.

\noindent{\bf Case 2.}
If $\pP$ is finite, there is some $p\in\pP$ and a sequence of naturals $e_i\to\infty$ such that,
for each $i$, $p^{e_i}$ divides $m_j$ for some $j$. Arguing as in Case 1,
we get an effective action of $(\ZZ/p^{e_i})^m$ on $X$ for each $i$.
By Lemma \ref{lemma:Minkowski} we may assume, replacing each $e_i$ (resp. $(\ZZ/p^{e_i})^m$)
by $e_i-C$ (resp. $(p^C\ZZ/p^{e_i})^m\simeq (\ZZ/p^{e_i-C})^m$)
for some suitable constant $C$ independent of $i$,
that the action of $(\ZZ/p^{e_i})^m$ on $X$ is $H^*$-trivial.

Let $\rho_i:(\ZZ/p^{e_i})^m\to T_H$ be the rotation morphism.
We are going to use the following result, which is \cite[Theorem 13.1]{Birkhoff}:

\begin{lemma}
Let $G=(\ZZ/p^d)^m$, and let $H\leq G$ be any subgroup.
There is an isomorphism $\phi:G\to G$ and integers $d_1\geq d_2\geq\dots\geq d_m$
such that
$$\phi(H)=\{(g_1,\dots,g_m)\in (\ZZ/p^d)^m\mid g_j\in p^{d-d_j}\ZZ/p^d\},$$
so that $\phi(H)\simeq \ZZ/p^{d_1}\times\dots\times\ZZ/p^{d_m}$.
\end{lemma}

For each $i$ the action of $(\ZZ/p^{e_i})^m$ on $X$ is given by a group morphism
$(\ZZ/p^{e_i})^m\to\Homeo(X)$. By the previous lemma, pre-composing this morphism
by a suitable isomorphism of $(\ZZ/p^{e_i})^m$, we may assume that
$$\Ker\rho_i=\{(g_1,\dots,g_m)\in (\ZZ/p^{e_i})^m\mid g_j\in p^{e_i-d_{j,i}}\ZZ/p^{e_i}\}$$
for some integers $d_{1,i}\geq d_{2,i}\geq\dots\geq d_{m,i}$. Let
$$t=\max\{j\mid d_{j,i}\text{ is unbounded as $i\to\infty$}\}.$$

Let $G_i:=\{(g_1,\dots,g_m)\in (\ZZ/p^{e_i})^m\mid g_1=\dots=g_t=0\}$.
Then $G_i\simeq(\ZZ/p^{e_i})^s$, where $s=m-t$, and $(\Ker \rho_i)\cap G_i$ is bounded.
We may thus apply Theorem \ref{thm:main-rotation-morphism} to the action of $G_i$ on
$X$ and conclude, picking $i$ big enough,
the existence of a continuous map $X\to T^k$, for some $k\geq s$, inducing an injective
map $H^*(T^k)\to H^*(X)$. It follows that $s\leq\tau(X)$.

Replacing, if necessary, the sequence of actions of $(\ZZ/p^{e_i})^m$ by a subsequence and relabelling,
we may assume that $d_{t,i}\to\infty$ as $i\to\infty$. Define $\delta_i:=d_{t,i}$. Then
$\Ker\rho_i$ has a subgroup $K_i$ isomorphic to $(\ZZ/p^{\delta_i})^t$, and $\delta_i\to\infty$.
Let $C'$ be the constant given by applying Theorem \ref{thm:fixed-points-kernel-rotation-morphism} to $X$.
Let $c:=\lceil\log_p (C')^{2/n}\rceil$, and define $\epsilon_i:=2\delta_in^{-1}-c$ for each $i$.
By Theorem \ref{thm:fixed-points-kernel-rotation-morphism}, any
subgroup of $K_i$ isomorphic to $\ZZ/p^{\delta_i}$ has a subgroup isomorphic
to $\ZZ/p^{\epsilon_i}$ which fixes at least one point of $X$. It follows that
$K_i':=p^{\delta_i-\epsilon_i}K_i\simeq (\ZZ/p^{\epsilon_i})^t$ has the weak
fixed point property (see the second paragraph in Subsection \ref{ss:proof-thm:main-big-prime-p}).

Let $C_F$ be the number defined in Subsection \ref{ss:proof-thm:main-big-prime-p}
using \cite[Theorem 1.6]{Mu2022}. For each $i$ there is a subgroup $K_i''\leq K_i$ satisfying
$[K_i':K_i'']\leq C_F$ and such that $K_i''$ fixes at least one point in $X$.
Choosing $i$ such that $p^{\epsilon_i}\geq p^2C_F$, we may assume that
$K_i''$ contains a subgroup isomorphic to $(\ZZ/p^2)^t$.
Arguing as in the proof of Case 1, using Lemma \ref{lemma:embedding-in-GL(n-s)},
we may deduce that
$K_i''$ is isomorphic to a subgroup of $\GL(n-s,\RR)$.
Since $K_i''$ has a subgroup isomorphic to $(\ZZ/p^2)^t$, Lemma \ref{lemma:cota-m-s}
implies that $t\leq (n-s)/2$. Now the proof of Case 2 can be concluded as in Case 1.

The previous arguments also imply the following result, which generalizes most of
Theorem \ref{thm:main-big-prime-p} to actions of groups of the form $(\ZZ/p^f)^r$.

\begin{theorem}
Let $X$ be an $n$-dimensional WLS manifold.
For any natural number $C_0$ there exists a number $C_1$ such that:
for any prime power $p^f\geq C_1$, and any effective and $H^*$-trivial
action of $G:=(\ZZ/p^f)^r$ on $X$ with rotation morphism $\rho:G\to T_H$
(where $H=H^1(X)$), there exist subgroups
$G_0\simeq(\ZZ/p^e)^{r-s}$ and $G_1\simeq (\ZZ/p^e)^s$ of $G$ satisfying:
\begin{enumerate}
\item $p^e\geq C_0$,
\item $\rho|_{G_1}$ is injective,
\item $X^{G_0}\neq\emptyset$, and each connected component of $X^{G_0}$
has dimension $\geq (n-s)/2$.
\end{enumerate}
\end{theorem}

Note that properties (2) and (3) imply that $G_0\cap G_1=0$.

\end{document}